\newif\ifpreprint
\preprinttrue

\ifpreprint
  \documentclass{article}
  \pdfoutput=1
  \usepackage[table]{xcolor}
  \usepackage{amsthm, amssymb}
  \newtheorem{theorem}{Theorem}
\else
  \documentclass[format=acmsmall, review=false, screen=true, table]{acmart}
\fi

\usepackage{booktabs}
\usepackage{multirow}
\usepackage{amsfonts, amsmath, algorithm, algorithmic}
\usepackage{url}
\usepackage{listings}
\usepackage{graphicx}
\usepackage{subcaption}
\usepackage{plotz}
\usepackage[utf8]{inputenc}
\usepackage{tikz}
\usetikzlibrary{matrix,backgrounds}
\usepackage{array}
\usepackage[normalem]{ulem}
\usepackage{pdflscape}
\usepackage{stmaryrd}

\usepackage[bb=dsserif]{mathalpha}
\usepackage{bm}

\newcommand\indicator{%
  \boldsymbol{\mathbb{1}}
}

\newtheorem{mydef}{Definition}
\theoremstyle{remark}
\newtheorem*{remark}{Remark}

\ifpreprint
  \title{Confidence Intervals for Stochastic Arithmetic}
  \author{Devan Sohier\thanks{devan.sohier@uvsq.fr} \hspace{1cm} Pablo de Oliveira Castro\thanks{pablo.oliveira@uvsq.fr} \\
    Université Paris-Saclay, UVSQ, LI-PaRAD
  \and
  François Févotte\thanks{francois.fevotte@triscale-innov.com}\\
  TriScale innov
  \and
  Bruno Lathuilière\thanks{bruno.lathuiliere@edf.fr}\\
  EDF R\&D -- PERICLES
  \and
  Eric Petit\thanks{eric.petit@intel.com}\\
  Intel Corp.
  \and
  Olivier Jamond\thanks{olivier.jamond@cea.fr}\\
  CEA
  }
\else
\acmJournal{TOMS}
\acmYear{2018}

  \title[Confidence Intervals for Stochastic Arithmetic]{Confidence Intervals for Stochastic Arithmetic}

  \author{Devan Sohier}
  \affiliation{\institution{Université Paris-Saclay, UVSQ, Li-PaRAD}
    \streetaddress{9 blvd. d'Alembert, bât. Rabelais}
    \city{Guyancourt}
    \country{France}
    \postcode{78280}
  }
  \email{devan.sohier@uvsq.fr}

  \author{Pablo de Oliveira Castro}
  \affiliation{\institution{Université Paris-Saclay, UVSQ, Li-PaRAD}
    \streetaddress{9 blvd. d'Alembert, bât. Rabelais}
    \city{Guyancourt}
    \country{France}
    \postcode{78280}
  }
  \email{pablo.oliveira@uvsq.fr}

  \author{François Févotte}
  \affiliation{\institution{TriScale innov}
    \streetaddress{Drahi - X Novation center}
    \city{Palaiseau}
    \country{France}
    \postcode{91128}
  }
  \email{francois.fevotte@triscale-innov.com}

  \author{Bruno Lathuilière}
  \affiliation{\institution{EDF R\&D -- PERICLES}
    \streetaddress{7 boulevard Gaspard Monge}
    \city{Palaiseau}
    \country{France}
    \postcode{91120}
  }
  \email{bruno.lathuiliere@edf.fr}

  \author{Eric Petit}
  \affiliation{\institution{Intel Corp.}
    \streetaddress{2 rue de Paris}
    \city{Meudon}
    \country{France}
    \postcode{92190}
  }
  \email{eric.petit@intel.com}

  \author{Olivier Jamond}
  \affiliation{\institution{CEA}
    \streetaddress{}
    \city{}
    \country{}
    \postcode{}
  }
  \email{olivier.jamond@cea.fr}
\fi

\begin{document}

\ifpreprint
\maketitle
\fi

\begin{abstract}
Quantifying errors and losses due to the use of Floating-Point (FP) calculations in industrial scientific computing codes is  an important part of the Verification, Validation and Uncertainty Quantification (VVUQ) process. Stochastic Arithmetic is one way to model and estimate FP losses of accuracy, which scales well to large, industrial codes. It exists in different flavors, such as CESTAC or MCA, implemented in various tools such as CADNA, Verificarlo or Verrou. These methodologies and tools are based on the idea that FP losses of accuracy can be modeled via randomness. Therefore, they share the same need to perform a statistical analysis of programs results in order to estimate the significance of the results.

In this paper, we propose a framework to perform a solid statistical analysis of Stochastic Arithmetic. This framework unifies all existing definitions of the number of significant digits (CESTAC and MCA), and also proposes a new quantity of interest: the number of digits contributing to the accuracy of the results. Sound confidence intervals are provided for all estimators, both in the case of normally distributed results, and in the general case. The use of this framework is demonstrated by two case studies of industrial codes: Europlexus and code\_aster.
\end{abstract}

\ifpreprint
\else
\maketitle
\fi

\section{Introduction}

Modern computers use the IEEE-754 standard for implementing floating point (FP)
operations. Each FP operand is represented with a limited precision.  Single
precision numbers have 23 bits in the fractional part and double precision numbers have
52 bits in the fractional part. This limited precision may cause numerical
errors~\cite{higham2002accuracy} such as absorption or catastrophic
cancellation which can result in loss of significant bits in the result.

Floating Point computations are used in many critical fields such as
structure, combustion, astrophysics or finance simulations. Determining
the precision of a result is an important problem. For well known algorithms, 
bounds of the numerical error can be derived mathematically for a given
dataset~\cite{higham2002accuracy}. 

The Stochastic Arithmetic field proposes automatic methods for estimating the
number of significant digits for complex programs. Two main methods have been proposed:
CESTAC~\cite{VIGNES2004} and Monte Carlo Arithmetic (MCA)~\cite{PARKER1997},
which differ in many subtle ways, but share the same general
principles. Numerical errors are modeled by introducing random perturbations at
each FP operation. This transforms the output of a given simulation code into realizations of a random variable. Performing a statistical analysis of a set of sampled outputs allows to stochastically approximate the impact of numerical errors on the code results.

In this paper, we propose a solid statistical analysis of Stochastic Arithmetic
that does not necessarily rely on the normality assumption and provides strong
confidence intervals. As stated in \cite[p.45]{PARKER1997}, \emph{``MCA is not committed to any assumption of normality. Sampling is an art and the right approach to sampling can settle very tricky problems.''} and, a little further: \emph{``MCA is not committed to a statistical inference model''}. Unlike traditional analysis which only considers the
number of reliable significant digits, this paper introduces a new quantity of
interest: the number of digits contributing to the accuracy of the final result.

Section~\ref{sec:review} reviews the stochastic arithmetic methods.
Section~\ref{sec:problem} formulates the problem rigorously and defines
several interesting scopes of study.
Then we provide in section~\ref{sec:normal} a statistical analysis for normal distributions  and in section~\ref{sec:general} for general distributions. Section~\ref{sec:experiments} validates our statistical framework on two industrial scientific computing codes: Europlexus and code\_aster. Finally, section~\ref{sec:limitations} discusses some of the remaining limitations of stochastic arithmetic methods, which should be addressed in future work. 

\section{Background on Stochastic Arithmetic Methods}
\label{sec:review}

Automatic methods for deriving bounds on round-off errors can be loosely categorized into two categories: exact methods and approximate methods.

Exact methods give a conservative and proven bound on the error of a computation. One well established exact method for deriving error bounds is Interval Arithmetic~\cite{moore2009introduction}, in which each real value in the algorithm is replaced by an interval that contains all the possible values of the computation. The operations are redefined to handle intervals operands and guarantee that the resulting interval provide rigorous bounds on the computation. Multiple software frameworks~\cite{rump1999intlab,revol2005motivations} for interval arithmetic have been released. Interval arithmetic have been applied to derive error bounds and optimize numerical methods~\cite{moore1979methods,kahan1996improbability}, linear algebra~\cite{hansen1965interval}, and physical simulation~\cite{dessombz2001analysis}. Because intervals are conservative, they tend to become overly large when the algorithm or control flow is complex. It is possible to refine the analysis by considering a union of interval subdivisions~\cite{hickey2001interval} or more sophisticate objects such as zonotopes~\cite{goubault2006static}; nevertheless in general for complex computer programs of thousands of code lines, deriving such an analysis is intractable. Exact approaches also include floating point proof assistants~\cite{de2006assisted,boldo2009combining,boldo2011flocq} which can derive semi-automatic certified proofs on floating point errors on small programs. 

On the other hand, approximate methods, do not provide deterministic bounds on the numerical error and may not always model exactly IEEE-754 behavior but are able to efficiently analyze large and complex programs, such as found in industrial codebases. For a more detailed comparison between stochastic and exact numerical analysis methods please refer to \cite{kahan1996improbability} and \cite[p.~71]{PARKER1997}.
This paper contributions focus solely on the stochastic arithmetic methods, which we will present in the following.

\subsection{Modeling accuracy loss using randomness}

When a program is run on an IEEE-754-compliant processor, the result of each
floating-point operation $x\circ y$ is replaced by a rounded value:
\mbox{$\text{round}(x\circ y)$}. For example, the default rounding mode for
IEEE-754 binary formats is given by:
\begin{align*}
  \text{round}(x)
   &= \lfloor x\rceil,
\end{align*}
where $\lfloor.\rceil$ denotes rounding to nearest, ties to even, for the
considered precision (\texttt{binary32} or \texttt{binary64}).

\medskip

When $x\circ y\neq \lfloor x\circ y\rceil$, \textit{i.e.} when $x\circ y$ is
not in the set $\mathbb{F}$ of representable FP values for the considered
precision, rounding causes a loss of accuracy. Stochastic arithmetic methods
model this loss of accuracy using randomness.

\subsubsection{CESTAC}\strut
The CESTAC method models round-off errors by replacing rounding operations with randomly rounded ones~\cite{VIGNES1974,chesneauxvignes}. The result of each FP operation \mbox{$x\circ y$} is substituted with
\mbox{$\text{random\_round}(x\circ y)$}, where $\text{random\_round}$ is a function
which randomly rounds FP values upwards or downwards equiprobably:
\begin{align*}
  \text{random\_round}(x)
  &=
    \left|\begin{array}{lcl}
            x &~~&\text{if $x\in\mathbb{F}$}\\[0.5em]
            \zeta \, \lfloor x\rfloor + (1-\zeta) \, \lceil x\rceil
            && \text{otherwise},
          \end{array}\right.
\end{align*}
where $\lfloor.\rfloor$ and $\lceil.\rceil$ respectively represent the downward and
upward rounding operations for the considered precision, and $\zeta$ is a random
variable such that \mbox{$\mathbb P[\zeta=0]=\mathbb P[\zeta=1]=\frac{1}{2}$}. CESTAC proposes different
variants by changing the probability distribution $\zeta$. 

\subsubsection{Monte Carlo Arithmetic}\strut
MCA can simulate the effect of different FP precisions by
operating at a virtual precision $t$. To model errors on a FP value $x$ at
virtual precision~$t$, MCA uses the noise function
\begin{align*}
  \text{inexact}(x) = x + 2^{e_x-t}\xi,
\end{align*}
where $e_x = \lfloor \log_2 \left|x\right| \rfloor +1 $ is the order of magnitude of $x$ 
and $\xi$ is a uniformly distributed random variable in the 
range $\left(-\frac{1}{2},\frac{1}{2}\right)$. During the MCA run of a given program, 
the result of each FP operation is replaced by a perturbed computation 
modeling the losses of accuracy~\cite{PARKER1997,verificarlo,frechtling2014tool}.
It allows to simulate the computation at
virtual precision $t$. Three possible expressions can be substituted to \mbox{$x\circ y$}, defining variants of MCA:
\begin{enumerate}
\item ``Random rounding'' only introduces perturbation on the output:\\\strut\hfill \(\text{round}(\text{inexact}(x\circ y))\)
\item ``Inbound'' only introduces perturbation on the input:\\\strut\hfill \(\text{round}(\text{inexact}(x)\circ \text{inexact}(y))\) 
\item ``Full MCA'' introduces perturbation on operand(s) and the result:\\\strut\hfill
\mbox{\(\text{round}(\text{inexact}(\text{inexact}(x)\circ \text{inexact}(y)))\)}
\end{enumerate}

\bigskip

In any case, using stochastic arithmetic, the 
result of each FP operation is replaced with a random variable modeling the
losses of accuracy resulting from the use of finite-precision FP
computations. Since the result of each FP operation in the program is in turn
used as input for the following FP operations, it is natural to assume that the
outputs of the whole program in stochastic arithmetic are random variables.

Stochastic Arithmetic methods run the program multiple times in order to
produce a set of output results (\textit{i.e.} a set of realizations or samples of the
random variable modeling the program output). The samples are then statistically
analyzed in order to assess the quality of the result.

\subsection{Estimating the result quality: significant digits}

Let us denote by $x$ the quantity computed by a deterministic 
numerical program. Different values can be defined for this result:
\begin{itemize}
\item $x_{\text{real}}$ is the value of $x$ that would be computed with an
  infinitely precise, real arithmetic;
\item $x_{\text{\sc ieee}}$ is the value that is computed by the program, when run
  on a computer that uses standard IEEE arithmetic with default rounding;
\item ${X_1, X_2, \ldots, X_n}$ are the values returned by $n$ runs of the
  program using stochastic arithmetic. These are seen as $n$ realizations of the
  same random variable $X$.
\end{itemize}

Figure~\ref{fig:schema} illustrates some of the quantities of interest that can
be useful to analyze the quality of the results given by the program.  The real
density of random variable $X$ is unknown, but some of its characteristics can
be estimated using $n$ sample values $(X_1, \ldots, X_n)$. In particular:
\begin{itemize}
\item the expected value $\mu = E[X]$ can be estimated by the empirical average
  value of ${X_i}$, $\hat\mu = \frac{1}{n} \; \sum_{i=1}^n X_i$;
\item the standard deviation $\sigma = \sqrt{E[(X - \mu)^2]}$ can be estimated
  by the empirical standard deviation, $\hat\sigma = \sqrt{\frac{1}{n-1} \; \sum_{i=1}^n \left(X_i-\hat\mu\right)^2}$.
\end{itemize}

\begin{figure}[ht]
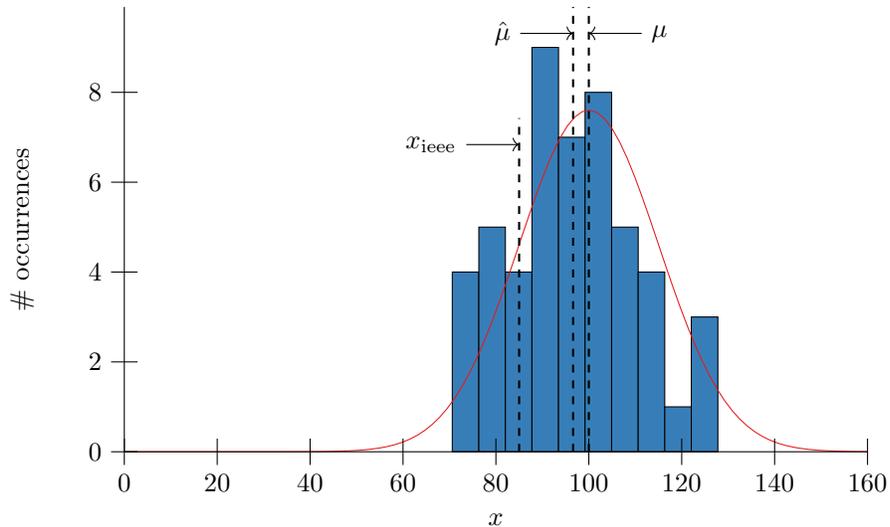

  \centering
  \plotz[width=\textwidth,height=20em,x=1.07828,y=0.78358]{schema}
  \caption{Schematic view of the various quantities of interest when evaluating
    the numerical quality of a result using stochastic arithmetic}
  \label{fig:schema}
\end{figure}

To estimate the numerical quality of the result, we would like to compute the
number of \emph{significant} bits. In the following we review the definitions
of \emph{significance} used in CESTAC and MCA. Then, we introduce the definition
that will be used in this paper.

\subsubsection{CESTAC definition of significant bits}\strut

In CESTAC, the average $\hat \mu$ of the small set of samples (usually three) is
taken as the computed result, and the analysis then estimates the accuracy of
this quantity, seen as an approximation of $x_{\text{real}}$.

\begin{mydef}\label{def:significant.cestac}
  With the notations defined above, the CESTAC number of exact significant
  bits~\cite{VIGNES2004} is defined as the number of bits in common between
  $x_{\text{real}}$ and $\hat\mu$:
  \begin{align*}
    s_ {\text{\sc cestac}}
    = -\log_2\left|\frac { x_{\text{real}} - \hat\mu}{x_{\text{real}}}\right|.
  \end{align*}
\end{mydef}

\bigskip

In order to estimate the number of exact significant digits, the CESTAC analysis is based on two hypotheses:
\begin{enumerate}
\item the distribution $X$ is normal, and
\item the distribution $X$ is centered on the real result $\mu = x_{\text{real}}$.
\end{enumerate}

Since $X$ is assumed normal, one can derive the following Student t-distribution interval with confidence $(1-\alpha)$:
\begin{align*}
\mu \in \left[\hat \mu \pm \frac{\tau_n \, \hat \sigma}{\sqrt{n}}\right],
\end{align*}
where $\tau_n$ is the $1-\frac{\alpha}{2}$ quantile of the Student distribution with $n-1$ degrees of freedom.

The maximum error between $\mu$ and $\hat\mu$ is bounded by this interval for a
normal distribution; it follows~\cite{Li2013Numerical} that an estimated lower
bound for the number of exact significant bits is given by
\begin{align}
  \label{eq:scestac}
  s_ {\text{\sc cestac}}
  = -\log_2\left|\frac { \mu - \hat\mu}{\mu}\right|
  \approx -\log_2\left|\frac { \mu - \hat\mu}{\hat \mu}\right|
  \geqslant \underbrace{-\log_2\left(\frac{\tau_n \, \hat \sigma}{\sqrt{n} \; |\hat \mu|}\right)}_{\hat s_{\text{\sc cestac}}}.
\end{align}

This definition suffers from a few shortcomings.  First, the two hypotheses,
while reasonable in many cases, do not always hold~\cite{chatelin88,kahan1996improbability}: 
Stott Parker shows that the
normality assumption of $X$ is not always true~\cite[p.~49]{PARKER1997} and it
is not necessarily centered on the real result. The robustness of CESTAC with respect to violations of these hypotheses is discussed in~\cite{chesneauxvignes}.

Second, and more important, the CESTAC definition of the number of significant digits may not necessarily be the most useful for the practitioner. Oftentimes, the objective of the numerical verification process consists in
evaluating the precision of the actual IEEE computer arithmetic.  CESTAC does
not evaluate the number of significant digits of the IEEE result but rather of the average
of the CESTAC samples. But in practice, $x_{\text{\sc ieee}}$ does not match
$\hat\mu$.

Last, with this definition, a problem clearly appears when considering the asymptotic behavior
of the bound: $\hat s_{\text{\sc cestac}} \xrightarrow[n\to+\infty]{} +\infty$.
 Increasing the number of samples arbitrarily increases the number of significant digits computed by CESTAC. On the one hand, this is expected because, according to the definition proposed, any computation is actually infinitely precise when $n\to\infty$ since the strong law of large numbers states that the empirical average is in this case almost surely the expected value. On the other hand, however, this asymptotic case also questions the pertinence of the CESTAC metric for the evaluation of the quality of the results produced by IEEE-754 computations. CESTAC is usually applied to three samples~\cite{chesneauxvignes}, the validity of the special case where $n=3$ is discussed in section~\ref{sec:normal-summary}.

\subsubsection{MCA definition of significant bits}\strut
\label{sec:parker-s}

In his study of MCA, Stott Parker proposes another definition for the number of
significant digits. Stott Parker lays this definition on the habits of biology
and physics regarding the precision of a measurement: if an MCA-instrumented
program is seen as a measurement instrument\footnote{In most applications, a measurement is modeled by a random variable following a normal distribution.}, then the number of significant
digits can be defined as the number of digits expected to be found in agreement
between successive runs/measurements.

\begin{mydef}\label{def:significant.mca}
  With the notations defined above, the MCA number of significant bits is defined as
  \begin{align*}
  s_{\text{\sc mca}} = -\log_{2}\left|{\frac{ \sigma}{\mu}}\right|.
  \end{align*}
\end{mydef}

\bigskip

This definition, which computes the magnitude of the coefficient of variation, is a form of signal to noise ratio: if most random samples share the same first digits, these digits
can be considered significant. On the contrary, digits varying randomly among
sampled results are considered noise. Another way of giving meaning to this
definition is to consider~$x_{\text{\sc ieee}}$ as one possible realization of
the random variable~$X$. As such, its distance to~$\mu$ is characterized by~$\sigma$. 
A problem with the MCA definition of significant bits is that it is empirical: the actual meaning of ``significance'' is not clearly laid, as well as the consequences one can draw from it.

\medskip

The MCA number of significant bits can be estimated by
\begin{align}
  \label{eq:smca}
  \hat s_{\text{\sc mca}} = -\log_{2}\left|{\frac{\hat \sigma}{\hat \mu}}\right|,
\end{align}
a quantity which can be computed regardless of any hypothesis on the distribution of $X$. However,
since the number of samples $n$ is finite, $\hat s_{\text{\sc mca}}$ is only an
approximation of the exact value $s_{\text{\sc mca}}$. And no confidence interval
is provided in order to help choose an appropriate number of samples. 
As shall be seen in section~\ref{sec:normal-summary}, this estimate is nevertheless a good basis when the underlying phenomenon is normal.

\bigskip

Our definition of significance generalizes Stott Parker's.  It requires a reference
which can be either a scalar value\footnote{Good choices for the scalar
  reference are $x_{\text{real}}$, $x_{\text{\sc ieee}}$ or $\mu_X$ depending on the aims of
  the study} or another random variable\footnote{Using a second random variable as reference allows comparing two versions of a program or two algorithms variants, more details are given in section~\ref{sec:problem}.}.
  Informally, in this paper, the significant
digits are the digits in common between the samples of $X$ and the
reference (up to a rounding effect on the last significant digit). Section~\ref{sec:problem} formalizes this definition in a
probabilistic framework and provides sound confidence intervals.

\subsection{Software tools presentation}
The experimental validation of the presented confidence intervals on synthetic and industrial use cases has been conducted thanks to the Verificarlo and Verrou tools which are presented in the next subsections.

\subsubsection{Verificarlo}

Verificarlo~\cite{verificarlo,verificarloproject} is an open-source tool based on the LLVM compiler framework replacing at compilation each floating point operation by custom operators. After compilation, the program can be linked against various backends~\cite{Chatelain2018veritracer, Chatelain2019automatic,Defour2020CustomPrecision}, including MCA to explore random rounding and virtual precision impact on an application accuracy.

Doing the interposition at compiler level allows to take into account the compiler optimization effect on the generated FP operation flow. Furthermore, it allows to reduce the cost of this interposition by optimizing its integration with the original code.

\subsubsection{Verrou}

Verrou \cite{fevotte2016,verrouproject} is an open-source floating point diagnostics tool. It is based on Valgrind~\cite{nethercote2007} to transparently intercept floating point operations at runtime and replace them by their random rounding counterpart.
The interposition at runtime allows to address
large and complex-code applications with no intervention of the end-user.

Verrou also 
provides two methods allowing to locate the origin of precision losses in the sources of the analyzed computing code.   
The first one is based on the code coverage comparison between
two samples. Discrepancies in the code coverage are good indicators of potential branch instabilities.
The second localization method leverages the delta-debugging algorithm~\cite{zeller2009} to perform a binary search to find a maximal scope for which MCA perturbations  do not produce errors or large changes in results. The remaining symbols (or lines if the binary is compiled with debug mode) are good candidates for correction.

\subsection{Synthetic example: Ill-conditioned linear system}
\label{sec:example}

To illustrate these methods in the following we use
a simple synthetic example proposed by Kahan~\cite{KAHAN66}: solving an
ill-conditioned linear system,
\begin{equation}
\label{eq::kahan_lin_system}
\left( \begin{array}{cc}
0.2161& 0.1441  \\
1.2969 & 0.8648  \\
\end{array} \right) x =
\left( \begin{array}{c}
0.1440   \\
0.8642   \\
\end{array} \right)
\end{equation}
The exact and IEEE \texttt{binary64} solutions of equation~\eqref{eq::kahan_lin_system} are:
\begin{equation}
 x_{\text{real}} =
\left( \begin{array}{c}
2   \\
-2   \\
\end{array} \right)
\qquad
 x_{\text{\sc ieee}} =
\left( \begin{array}{c}
1.9999999958366637 \\
-1.9999999972244424 \\
\end{array} \right)
\end{equation}

To keep the example simple, the floating-point solution $x_{\text{\sc ieee}}$ has been
obtained by solving the system with the naive C implementation of Cramer's formula
in double precision, as shown in listing~\ref{lst:cramer}.

\begin{lstlisting}[language=C, caption={Solving 2x2 system $a.x = b$ with Cramer's rule}, label={lst:cramer}]
void solve(const double a[4], const double b[2], double x[2]) {
  double  det = a[0] * a[3] - a[2] * a[1];
  double det0 = b[0] * a[3] - b[1] * a[1];
  double det1 = a[0] * b[1] - a[2] * b[0];
  x[0] = det0/det;
  x[1] = det1/det;
}
\end{lstlisting}

The condition number of the above system is approximately $2.5\times 10^8$, therefore we expect to
lose at least $\log_2(2.5\times 10^8) \approx 28$ bits of accuracy or,
equivalently, $8$ decimal digits. By comparing the IEEE and exact values, we
see that indeed the last $8$ decimal digits differ. The number of common bits
between $x_{\text{real}}$ and $x_{\text{\sc ieee}}$ is given by
\begin{align*}
  s_{\text{\sc ieee}}
  = -\log_2\left\vert\frac{x_{\text{real}}-x_{\text{\sc ieee}}}{x_{\text{real}}}\right\vert
  \approx \left(\begin{array}{c}28.8 \\ 29.4\end{array}\right).
\end{align*}

Now let us use MCA to estimate the number of significant digits.  We compile the
above program with Verificarlo~\cite{verificarlo} which transparently replaces
every FP operation by its noisy MCA counterpart. Here a virtual precision of 52
is used to simulate roundoff errors. Then, we run the
produced binary \mbox{$n = 10\,000$} times and observe the resulting output distribution $X$.

Both $X[0]$ and $X[1]$ are normal with high Shapiro-Wilk test p-values 73~\% and
74~\% respectively.\footnote{Interestingly $X[0]$ fails the Anderson-Darling
  test, 27~\% p-value, due to some anomalies on the tail.}
Figure~\ref{fig:normality-cramer1} shows the distribution and
quantile-quantile~(QQ) plots for $X[0]$, for which the empirical average and
standard deviation are given by
\begin{align*}
  \hat\mu &\approx 1.99999999909,\\
  \hat\sigma &\approx 5.3427\times 10^{-9}.
\end{align*}

\begin{figure}
\centering
\includegraphics[width=\linewidth]{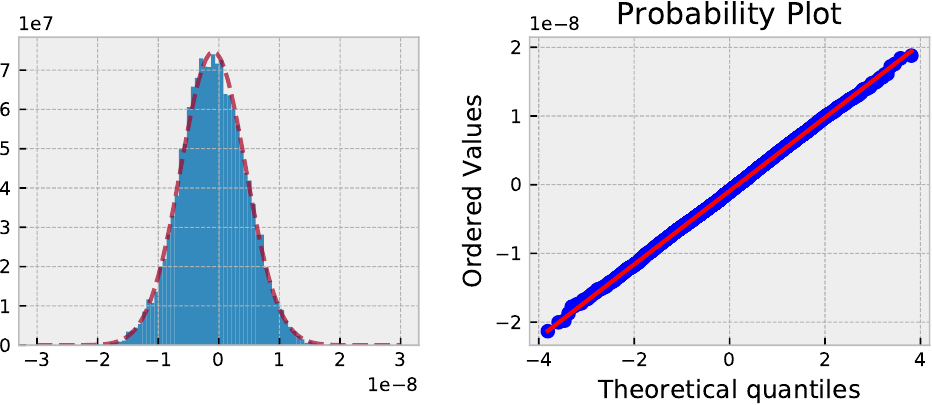}
\caption{Normality of Cramer $X[0]$ sample}
\label{fig:normality-cramer1}
\end{figure}

Using Stott Parker's formula~\eqref{eq:smca} to compute
$\hat s_{\text{\sc mca}}$ for $X[0]$, we get a figure close to the expected
value $28.8$:
\begin{align}
  \label{eq:shat}
  \hat s_{\textrm{MCA}} = -\log_{2}{\left|\frac{\hat \sigma}{\hat \mu}\right|} \approx 28.48.
\end{align}
But how confident are we that $\hat s_{\text{\sc mca}}$ is a good estimate of
$s_{\text{\sc mca}}$? Could we have used a smaller number of samples and still
get a reliable estimate of the results quality?

\bigskip

On the other hand, using these $n=10\,000$ samples to compute the CESTAC lower
bound defined in equation~\eqref{eq:scestac} with confidence 95~\% gives
\begin{align}
  \label{eq:sbound}
  \hat s_{\text{\sc cestac}}^{(10\,000)} = -\log_2\left(\frac{\tau_n\,\hat \sigma}{\sqrt{n} \, |\hat \mu|}\right) \approx 34.2,
\end{align}
which is a clear overestimation of the quality of the IEEE result, but also of
the CESTAC result, since the number of bits in common between the real result
and the sample average is given by
\begin{align*}
  s_{\text{\sc cestac}}^{(10\,000)}
  &= -\log_2\left\vert\frac{x_{\text{real}}[0]-\hat\mu}{x_{\text{real}}[0]}\right\vert
    \approx 31.0.
\end{align*}
Such a large $n$ exhibits the bias between $\mu$ and $x_{\text{real}}[0]$, invalidating the CESTAC hypotheses. 
In practice, CESTAC implementations such as CADNA use $n=3$, a choice
of which the validity is discussed in section~\ref{sec:normal-summary}. For the Cramer benchmark, computing $\hat s_{\text{\sc cestac}}$ for
only 3 samples of X[0] yields a conservative estimate:
\begin{align*}
\hat s_{\text{\sc cestac}}^{(3)} &\approx 27.5 \leqslant s_{\text{\sc cestac}}^{(3)} \approx 28.5.
\end{align*}

\bigskip

In the following, we present a novel probabilistic formulation to get a confidence interval for the number of significant bits with and without assumption of normality.

\section{Probabilistic accuracy of a computation}
\label{sec:problem}

\subsection{Definitions}

We consider one output of a program performing FP operations as a random variable $X$.
The output is a random variable either because the program is inherently nondeterministic or because we are artificially introducing numerical errors through MCA, CESTAC, or another stochastic arithmetic model.
We want to study how the probabilistic properties of the computation impact its accuracy.
The real distribution of $X$ is unknown but we can approximate it with  $n$ samples, $X_1 \ldots X_n$.

The accuracy of a result must be defined against a reference value. If a real mathematical result is known, it is a natural choice. If the program is deterministic when executed in IEEE arithmetic, the IEEE result is one straightforward choice for the reference value. If the program is nondeterministic, one can also choose as reference, the empirical average of $X$. Finally a third option consists in computing the accuracy against a second random variable $Y$, which allows computing the accuracy between runs of the same program or allows finding the accuracy between two different programs, such as when comparing two different versions or implementations of an algorithm.
We will write the reference value $x$ when it is a constant and $Y$ when it is another random variable. 

Four types of studies can be led, depending on whether we are interested in absolute or relative error, and whether we have a reference value. For each study we can model the errors as a random variable $Z$ defined as follows

\begin{center}
\begin{tabular}{lll}
\toprule
				  &reference $x$ & reference $Y$ \\
absolute precision& $Z=X-x$       & $Z=X-Y$       \\
relative precision& $Z=X/x-1$     & $Z=X/Y-1$    \\
\bottomrule
\end{tabular}
\end{center}

We have reduced the four types of problems to study the probabilistic properties of $Z$ whose error distribution represents the error of a computation in a broad sense. If $X$ is unbiased \emph{wrt.} the reference, $\mathbb E[Z]=0$ (a constant reference $x$ can be for instance the exact result of a computation that introduces a bias, for instance when a division by a value close to 0 occurs; in such cases we observe $\mathbb E[X]\neq x$).

To define the significance of a digit we use Stott Parker's $\frac12$ulsp algorithm~\cite[p.~19]{PARKER1997}. The significant bit is at the rightmost position at which the digits differ by less than one half unit in the last place. That is to say, two values $x$ and $y$ have $s$ significant digits\footnote{A non-strict inequality is often used in this definition. Under the assumption made by most previous works that the underlying distribution is continuous (\emph{e.g.} normal), both definitions agree. We chose a strict inequality that better fits with the notion of contributing digit that we introduce.} iff
\begin{align}
|x-y| <& \frac{1}{2} \times 2^{e_y-s} = 2^{-s+(e_y-1)} &\text{(scaled absolute error)} \nonumber\\ 
|x/y - 1| <& \frac{1}{2} \times 2^{1-s} = 2^{-s}   &\text{(relative error)} \label{eq:ulsp}
\end{align}
Without loss of generality, to unify the definition for the relative and scaled absolute cases, in the following sections 
we assume $e_y=1$. When working with absolute errors, one should therefore shift the number of digits by $(e_y-1)$, the normalizing term\footnote{When $Y$ is a random variable, we choose $e_Y = \lfloor \log_2 \left|E[Y]\right| \rfloor + 1$.}.

The first quantity of interest is the \textbf{probability that the result is significant up to a given
  bit} for a stochastic computation. The stochastic computation can be for example a program instrumented with CESTAC or MCA. By generalizing equation~\ref{eq:ulsp} to random variables, we define the probability of the $k$-th digit being significant as $\mathbb P\left(|Z| < 2^{-k}\right)$.

\begin{mydef}\label{def:significant}
For a given stochastic computation, the $k$-th bit of $Z$ is said to be significant with probability $p$ if 
\begin{align*}
  \mathbb P\left(\vphantom\sum |Z| < 2^{-k}\right) \geqslant p.
\end{align*}

The number of significant digits in $Z$ with probability $p$ is defined as the largest number $s_{\text{\rm sto}}\in\mathbb{R}$ such that 
\begin{align*}
  \mathbb P\left(\vphantom\sum |Z| < 2^{-s_{\text{\rm sto}}}\right) \geqslant p.
\end{align*}
\end{mydef}
\bigskip

Note that, by definition, if the $k$-th bit of $Z$ is significant with probability $p$, then any bit of rank $k^\prime \leqslant k$ is also significant with probability $p$. In the remainder of this paper, when not otherwise specified, the simple notation $s$ will refer to the $s_{\text{sto}}$ notion defined above.

The second quantity we will consider is the \textbf{probability that a given bit contributes to the precision of the result}: even if a bit on its left is already wrong, a bit can either improve the result precision, or deteriorate it. As noted in~\cite[p.45]{PARKER1997}: \emph{''In other words, in inexact values it can be worthwhile to carry a nontrivial number $r$ of random least significant bits''}. Because the expected result of $Z$ is 0, a bit will improve the accuracy if it is 0 and deteriorate it if it is 1.

\begin{mydef}
  The $k$-th bit of $Z$ is said to be contributing with probability~$p$ if and only if it is~$0$ with this probability, \textit{i.e.} if and only if
  \begin{align*}
  \mathbb P\left(\left \lfloor{2^{k}\left|Z\right|}\right \rfloor \text{\rm is even}\right) \geqslant p.
  \end{align*}
\end{mydef}
\bigskip

Now, the $k$-th bit of $Z$ is~0 if and only if there exists an integer $i$ such that,
\begin{align}
  & \left \lfloor{2^{k}\left|Z\right|}\right \rfloor = 2i  \nonumber \\ 
\Leftrightarrow\mkern40mu & 2i \leqslant 2^{k}\left|Z\right| < 2i+1  \nonumber \\ 
\Leftrightarrow\mkern40mu & 2^{-k}(2i) \leqslant \left|Z\right| < 2^{-k}(2i+1). \label{eq:contribution-formula}
\end{align}

One should note that the notions of significance and contribution are distinct, but related: if there are $s$~significant bits with probability~$p$, then all bits at ranks~$c\leqslant s$ are contributing, with probability~$p$. Indeed,
\begin{align*}
& \mathbb P\left(|Z| < 2^{-s}\right) \geqslant p \\
\Rightarrow\mkern40mu & \forall c \leqslant s, \; \mathbb P\left(2^{c}\,|Z| < 1\right) \geqslant p \\ 
\Rightarrow\mkern40mu & \forall c \leqslant s, \; \mathbb P\left(\left\lfloor\vphantom\sum 2^{c}\,|Z| \right\rfloor = 2 \times 0\right) \geqslant p.
\end{align*}

However, the $k$-th bit of $Z$ being contributing with probability $p$ does not imply that all bits at ranks~$k^\prime<k$ are also contributing\footnote{Although, it is the case for example when $Z$ follows a Gaussian distribution.}. This prevents the definition of such a notion as the number of contributing bits.

\bigskip

In the following, we study these two properties, \textbf{significant} and \textbf{contributing} bits, under the normality assumption (section~\ref{sec:normal}) and in the general case (section~\ref{sec:general}).

\subsection{Summary of the results}

In the remainder of this paper, we obtain the following results, unifying the various definitions of significance seen above and generalizing their validity to the non-Gaussian case.

\medskip

{\bf Under the Centered Normality Hypothesis (CNH)}, \textit{i.e.} if $X$ follows a Gaussian law centered around the reference value or, equivalently, if $Z$ follows a Gaussian law centered around $0$, it is shown in section~\ref{sec:significant} that a lower bound of the number of significant digits~$s_{\text{sto}}$ (as introduced in definition~\ref{def:significant}) with probability~$p$ and confidence level~$1-\alpha$ is given by
\begin{align*}
  \hat{s}_{\text{\sc cnh}}
  = -\log_2\left(\hat\sigma_Z\right)
    - \left[\frac{1}{2}\log_2\left(\frac{n-1}{\chi_{1-\alpha/2}^2}\right)
  +\log_2\left(F^{-1}\left(\frac{p+1}2\right)\right)\right],
\end{align*}
where $\hat\sigma_Z$ denotes the standard deviation of $n$ samples of $Z$, and all notations are relatively standard and are introduced in section~\ref{sec:normal}. Furthermore, the following results are established:
\begin{itemize}
\item $\hat{s}_{\text{\sc cnh}}$ can be computed simply by shifting the usual
  $\hat{s}_{\text{\sc mca}}$ estimator (introduced in
  definition~\ref{def:significant.mca}) by a certain number of bits, depending
  on $n$, $p$ and $\alpha$ and tabulated in Appendix~\ref{sec:abaque},
  Table~\ref{tab:shift};
\item Consequently, $\hat{s}_{\text{\sc mca}}$ can be re-interpreted in this
  framework; it is a lower bound on the number of significant bits with a certain probability and a given confidence level;
\item Although $\hat{s}_{\text{\sc cestac}}$ was not originally meant to do so, it can also be re-interpreted as an estimate of $s_{\text{sto}}$. For example in the CADNA case, where the number of samples is set to $n=3$ and a 95\% confidence level is used, $\hat s_{\text{\sc cestac}}$ estimates $s_{\text{sto}}$ with probability $p=0.3$.
\end{itemize}

\bigskip

{\bf In the general case}, when no assumption can be made about the
distributions of $X$ or $Z$, we introduce the Bernoulli significant bits estimator
\begin{align*}
  \hat s_{\text{\sc b}} = \max \left\{ k \in \{ 1, 2, \ldots ,53\} \text{ such that }
  \forall i\in\left\{ 1, 2, \ldots ,n\right\}, \left|Z_i\right| < 2^{-k} \right\}.
\end{align*}
It is shown in section~\ref{sec:general} that $\hat s_{\text{\sc b}}$ provides a
sound lower bound for $s$, provided that the number of samples~$n$ is chosen
accordingly to the desired probability~$p$ and confidence level~$1-\alpha$:
\begin{align*}
  n \geqslant \frac{\ln(\alpha)}{\ln(p)}.
\end{align*}
The required number of samples is tabulated in Appendix~\ref{sec:abaque},
Table~\ref{tab:nsamples}

\bigskip

{\bf Regarding contributing bits}, it is proved in
section~\ref{sec:norm.contrib-bits} that all bits of rank
$k\leqslant\hat c_{\text{\sc cnh}}$ are contributing with probability~$p$ and confidence
level~$1-\alpha$ under the centered normality hypothesis, with
\begin{align*}
  \hat c_{\text{\sc cnh}} =
  -\log_2(\hat\sigma)
  -\left[\frac12\log_2\left(\frac{n-1}{\chi^2_{1-\alpha/2}}\right)
  +\log_2\left(p-\frac{1}{2}\right)
  +\log_2\left(2\sqrt{2\pi}\right)\right].
\end{align*}
In section~\ref{sec:general}, it is proved that in the general case, the $k$-th bit of $Z$ is contributing with probability~$p$ and confidence level~$1-\alpha$ if
\begin{align*}
  \forall i\in\left\{ 1, 2, \ldots,n\right\}, \left\lfloor 2^{k}\left|Z_i\right| \right\rfloor \textrm{is even},
\end{align*}
provided that the number of samples~$n$ follows the rules described above.

\bigskip

Impatient readers may skip to section~\ref{sec:experiments}, where these results
are used for the analysis of industrial calculations. The more mathematically
inclined reader is of course encouraged to follow along
sections~\ref{sec:normal} and~\ref{sec:general}, where these results are
detailed and proved.

\section{Accuracy under the Centered Normality Hypothesis}
\label{sec:normal}
In this section we consider that $Z$ is a random variable with normal distribution $\mathcal N(0, \sigma)$. In practice, we only know an empirical standard deviation~$\hat\sigma$, measured over $n$~samples. Because $Z$ is normal, the following confidence interval~\cite[p.~282]{saporta2006probabilites} with confidence $1 - \alpha$ based on the $\chi^2$ distribution with $(n-1)$ degrees of freedom is sound~\footnote{
This interval is bilateral. If we were only interested in a lower bound for significant and contributing bits we could use the unilateral bound $\sigma^2 \leqslant (n-1)\hat\sigma^2/{\chi_{1-\alpha}^2}$.}:

\begin{align}
  \dfrac{(n-1)\hat\sigma^2}{\chi_{\alpha/2}^2} \leqslant \sigma^2 \leqslant \dfrac{(n-1)\hat\sigma^2}{\chi_{1-\alpha/2}^2}.
  \label{eq:chi2}
\end{align}

It is important to note that $\sigma$ is the standard deviation of $Z$ and not of $X$.
For example, when taking a second independent random variable $Y$ as reference, if $X$ and $Y$ both follow a distribution $\mathcal N(\mu, \sigma')$, $Z=X-Y$ follows $\mathcal N(0, \sqrt2\sigma')$. 

\subsection{Significant bits}
\label{sec:significant}

The theorem below is a more precise restatement of Stott Parker's Theorem 1: \emph{``the difference in the orders of magnitude of the mean $\mu$ and the standard deviation $\sigma$ measures the number of significant digits of $X$ (if $\mu\neq0$, $\sigma\neq 0$).''} We define the notion of ``measuring the number of significant digits'' as the estimation of the probability that a given bit is significant at a given confidence level. We then prove that the number of significant bits is given by $-\log_2\frac\mu\sigma$ as exposed by Stott Parker (since in a relative precision analysis, $\sigma_Z=\frac{\sigma_X}{x}=\frac{\sigma}{\mu}$ if $X$ is normal and centered at the reference value), but adjusted by a quantity that depends only on the target probability and confidence level (so, constant \emph{wrt} the sample). This new formulation allows to assess the consequences of taking the considered bits into account or not.

\begin{theorem}\label{ThPsig} For a normal centered error distribution $Z \sim \mathcal N(0, \sigma)$, the $s$-th bit is significant with probability 
\label{th:significant}
\begin{align*}
p_s = 2F\left(\frac{2^{-s}}\sigma\right)-1,
\end{align*} 
with $F$ the cumulative distribution function of the normal distribution with mean 0 and variance 1.
\end{theorem}
\begin{proof}

The probability that the $k$-th bit is significant is $\mathbb P\left[\left|Z\right| < 2^{-k}\right]=\mathbb P\left[Z<2^{-k}\right]-\mathbb P\left[Z<-2^{-k}\right]$. Now $\mathbb P\left[Z<-2^{-k}\right]=1-\mathbb P\left[Z<2^{-k}\right]$ by symmetry of the normal distribution, so that $\mathbb P\left[\left|Z\right|<2^{-k}\right]=2\mathbb P\left[Z<2^{-k}\right]-1$. Therefore, 

$$\mathbb P\left[\left|Z\right|<2^{-k}\right]=2\mathbb P\left[\frac{Z}\sigma<\frac{2^{-k}}\sigma\right]-1=2F\left(\frac{2^{-k}}\sigma\right)-1.$$

\end{proof}
The number of significant digits with probability $p$ is $s$ such that $2F\left(\frac{2^{-s}}\sigma\right)-1=p$, \textit{i.e.} \mbox{$F\left(\frac{2^{-s}}\sigma\right)=\frac{p+1}2$} $\Leftrightarrow$ \mbox{$\frac{2^{-s}}\sigma=F^{-1}\left(\frac{p+1}2\right)$}, so that

\begin{align*}
s=-\log_2\left(\sigma\right)-\log_2\left(F^{-1}\left(\frac{p+1}2\right)\right).
\end{align*}

The above formula is remarkable because, whatever $\sigma$, the confidence interval to reach a given probability is constant and can be computed from a table for $F^{-1}$. 
Therefore, one just needs to subtract a fixed number of bits from $-\log_2(\sigma)$ to reach a given probability, as illustrated in figure~\ref{fig:significant-bits}.

In practice, only the sampled standard deviation~$\hat\sigma$ can be measured, but it can be used to bound~$\sigma$ thanks to the the $\chi^2$ confidence interval in equation~\eqref{eq:chi2}. This allows computing a sound lower bound~$\hat s_{\text{\sc cnh}}$ on the number of significant digits in the Centered Normality Hypothesis:
\begin{align}
  s
  &\geqslant\underbrace{
    -\log_2\left(\hat\sigma\right)
    -\underbrace{
    \left[\frac{1}{2}\log_2\left(\frac{n-1}{\chi_{1-\alpha/2}^2}\right)
    +\log_2\left(F^{-1}\left(\frac{p+1}2\right)\right)\right]}_{\delta_{\text{\sc cnh}}}}_{\hat{s}_{\text{\sc cnh}}}. \label{eq:scnh}
\end{align}

Again, this formula is interesting since $\hat s_{\text{\sc cnh}}$ can be determined by just measuring the sample standard deviation~$\hat\sigma$ and shifting $-\log_2(\hat\sigma)$ by a value $\delta_{\text{\sc cnh}}$, which only depends on a few parameters: the size of the sample~$n$, the confidence $1-\alpha$ and the probability~$p$. Some values for this shift are tabulated in appendix~\ref{sec:abaque}, table~\ref{tab:shift}. This is an improvement over the proposition of~\cite[p.23]{PARKER1997} to use a confidence interval on the estimate of $\mu$. Instead, we propose a confidence interval directly on the quantity of interest, namely, the number of significant digits.

\begin{figure}
\centering
\includegraphics[width=.8\linewidth]{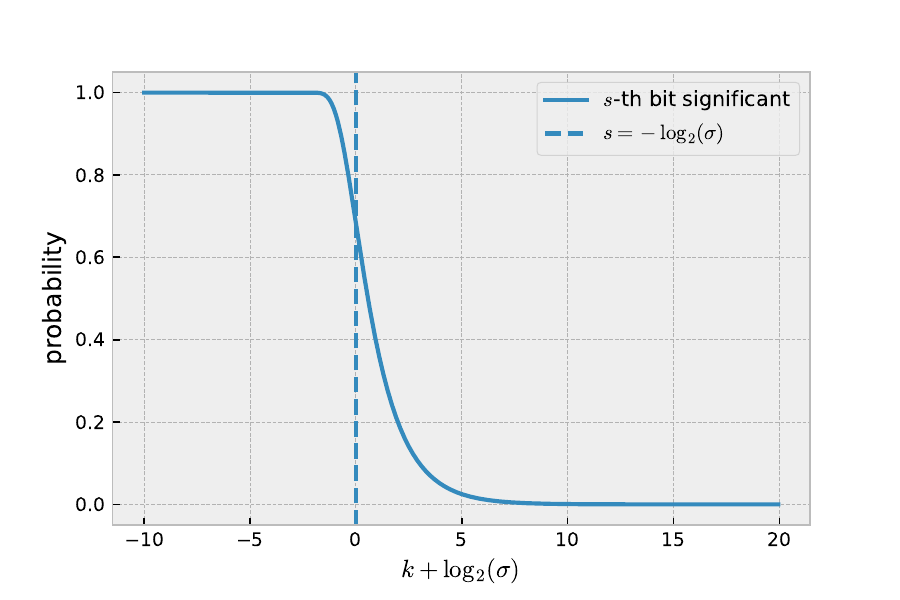}
\caption{ Profile of the significant bit curve: when the dashed line is positioned on the $-\log_2\sigma$ abscissa, the curve corresponds to the probability that the result is significant up to a given bit.}
\label{fig:significant-bits}
\end{figure}

\begin{figure}
\centering
\includegraphics[width=.48\linewidth]{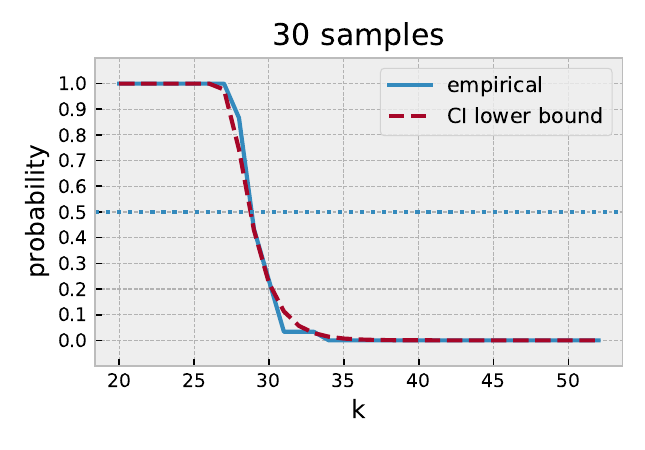}
\includegraphics[width=.48\linewidth]{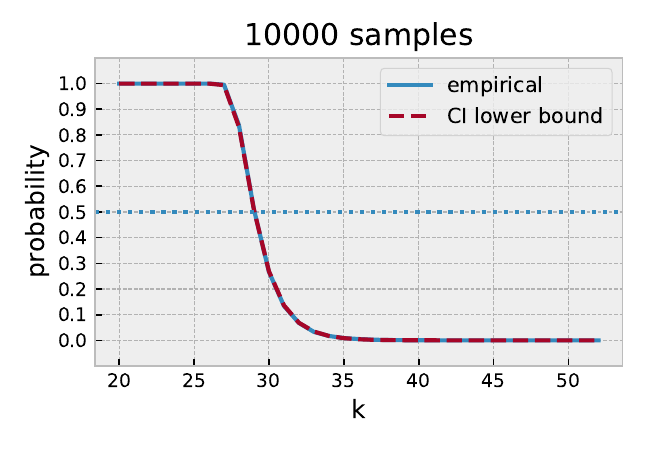}\caption{Significant bits for Cramer $x[0]$ variable computed under the normal hypothesis using 30 and 10000 samples. The Confidence Interval (CI) lower bound is computed by using the probability of theorem~\ref{th:significant} and bounding $\sigma$ with a 95\% Chi-2 confidence interval.}
\label{fig:cramer-0-sb-predicted}
\end{figure}

\paragraph{Application}
Let us consider the $X[0]$ variable from the the ill-conditioned Cramer system from section~\ref{sec:example}. Statistical tests did not reject the normality hypothesis for $X[0]$. Here we would like to compute the number of significant digits relative to the mean of the sample with a 99~\% probability. Following section~\ref{sec:problem}, we consider the relative error, $Z=\frac{X[0]}{\hat\mu}-1 \rightarrow \mathcal N(0, \sigma)$. Here $\sigma$ will be estimated from $\hat \sigma$ with the $\chi^2$ 95~\% confidence interval presented in equation~\eqref{eq:chi2}. Computing $\delta_{\text{\sc cnh}}$ for $n=10~000$, $p=0.99$ and $1-\alpha=0.95$ (or reading it in table~\ref{tab:shift}), yields $\delta_{\text{\sc cnh}} \approx 1.4$. Recalling the sampled measurements from section~\ref{sec:example}, we get $-\log_2(\hat\sigma) \approx 28.5$.

Therefore, at least $28.5-1.4=27.1$~bits are significant, with probability 99~\% at a 95~\% confidence level.
Figure~\ref{fig:cramer-0-sb-predicted} shows that the proposed confidence interval closely matches the empirical probability on the $X[0]$ samples. When the number of samples increases, the confidence interval tightness increases.

\subsection{Contributing Bits}\label{sec:norm.contrib-bits}

\begin{figure}
\centering\includegraphics[width=.9\linewidth]{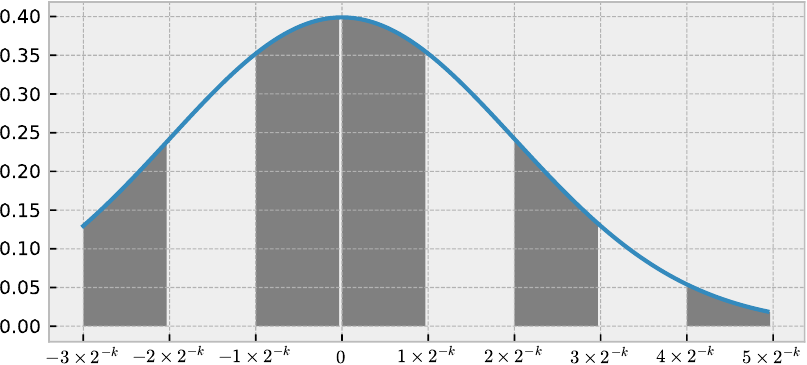}
\caption{Normal curve; the gray zones correspond to the area where the $k$-th bit contributes to make the result closer to 0 (whatever the preceding digits).}
\label{fig:norm-contributing}
\end{figure}

In the previous section we computed the number of significant bits. Now we are interested in the number of contributing bits: even if a bit is after the last significant digit, it may still contribute partially to the accuracy if it brings the result closer to the reference value. 

The theorem below gives an approximation of the number of contributing bits which has the same property as theorem~\ref{ThPsig}: this approximation computes the number of bits to shift from $-\log_2(\hat\sigma)$ to obtain the contributing bits based on the same few parameters (sample size $n$, confidence $1-\alpha$ and probability $p$), and the shift being independent of $\hat{\sigma}$.

\begin{theorem} For a normal centered error distribution $Z \sim \mathcal N(0, \sigma)$, 
when $\frac{2^{-c}}\sigma\rightarrow 0$, the $c$-th bit contributes to the result accuracy with probability 
\begin{align*}
p_c\sim \frac{2^{-c}}{2\sigma\sqrt{2\pi}} + \frac{1}{2}.
\end{align*}
\end{theorem}
\begin{proof}
As shown in equation~\ref{eq:contribution-formula},  the $k$-th bit of $Z$ contributes if and only if there exists an integer~$i\geqslant0$ such that $2^{-k}(2i) \leqslant |Z| < 2^{-k}(2i+1)$, \emph{i.e.} there exists an integer $i\geqslant 0$ such that $2^{-k}(2i) \leqslant Z < 2^{-k}(2i+1)$ or $i<0$ such that $2^{-k}(2i) < Z \leqslant 2^{-k}(2i+1)$. $Z$ being continuous, $\mathbb P[2^{-k}(2i) < Z \leqslant 2^{-k}(2i+1)]=\mathbb P[2^{-k}(2i) \leqslant Z < 2^{-k}(2i+1)]$.

For a normal centered $Z$ distribution, this inequality corresponds to the gray stripes in figure~\ref{fig:norm-contributing}. Let us write the integral of one stripe as, 

\begin{align*}
u_{(k, 2i)}=\mathbb P\left[2^{-k}(2i)\leqslant Z<2^{-k}(2i+1)\right]=\int_{2^{-k}(2i)}^{2^{-k}(2i+1)}f(x)\;\text{d}x,
\end{align*}

where $f(x)=\frac{e^{-\frac{x^2}{2\sigma^2}}}{\sigma\sqrt{2\pi}}$ is the probability density function of $\mathcal N(0, \sigma)$. The probability of contribution for the $k$-th bit, $p_k$, is therefore

\begin{align*}
&p_k  = \sum_{i\in\mathbb{Z}} u_{(k, 2i)} = 1-\sum_{i\in\mathbb{Z}} u_{(k, 2i+1)} \\
\Leftrightarrow\qquad&2p_k = 1 + \sum_{i\in\mathbb{Z}} {\left(u_{(k, 2i)}-u_{(k, 2i+1)} \right)} \\
\Leftrightarrow\qquad& p_k = \frac{1}{2} + \sum_{i\geqslant 0}{\left(u_{(k, 2i)} - u_{(k, 2i+1)}\right)}= \frac{1}{2} + \sum_{i\geqslant 0}(-1)^i u_{(k,i)} &&\textrm{ (by symmetry of $f$).} \\
\end{align*}

Now, according to the \emph{trapezoidal rule}, there exists $\xi_{(k, i)}$ in the interval $I_k^i = [2^{-k}(i), 2^{-k}(i+1)]$ such that

\begin{align*}
u_{(k, i)} = \int_{2^{-k}i}^{2^{-k}(i+1)}f(x) \; \text{d}x = 2^{-k}\left(\frac{f\left(2^{-k}i\right)+f\left(2^{-k}(i+1)\right)}{2}\right)+\frac{\left(2^{-k}\right)^3}{12}f''(\xi_{(k, i)}).
\end{align*}

Introducing $v_{(k, i)}=2^{-k}\left(\frac{f\left(2^{-k}i\right)+f\left(2^{-k}(i+1)\right)}{2}\right)$ and $w_{(k, i)}=\frac{\left(2^{-k}\right)^3}{12}f''(\xi_{(k, i)})$, we have $u_{(k, i)}=v_{(k, i)}+w_{(k, i)}$, and 
\begin{align*}p_k = \frac{1}{2} + \sum_{i\geqslant 0}(-1)^i v_{(k,i)}+ \sum_{i\geqslant 0}(-1)^i w_{(k,i)}.\end{align*}

Now we compute the alternate series $\sum_{i\geqslant 0} (-1)^i v_{(k, i)}$. Since the series is alternate and its terms tend to 0, all term cancellations are sound. All trapezoidal terms cancel, except the first half-term:

\begin{align*} 
\sum_{i\geqslant 0} (-1)^i v_{(k, i)} &=\sum_{i\geqslant 0} (-1)^i 2^{-k}\left(\frac{f\left(2^{-k}i\right)+f\left(2^{-k}(i+1)\right)}{2}\right)\\
&=2^{-k-1}\sum_{i\geqslant 0} (-1)^i \left(f\left(2^{-k}i\right)+f\left(2^{-k}(i+1)\right)\right)\\
&=2^{-k-1}f(0)=\frac{2^{-k}}{2\sigma\sqrt{2\pi}}.
\end{align*}

We can also sum the error terms $w_{(k,i)}$. We have:

\begin{align}
\frac{\left(2^{-k}\right)^3}{12}\min_{\xi \in I_k^i} f''(\xi)\leqslant w_{(k,i)}\leqslant\frac{\left(2^{-k}\right)^3}{12}\max_{\xi \in I_k^i} f''(\xi) \label{eq:errors}.
\end{align}

The second derivative $f''(\xi)=\frac{(\xi^2/\sigma^4-1/\sigma^2)}{\sigma\sqrt{2\pi}}e^{-\frac{\xi^2}{2\sigma^2}}$ is increasing between 0 and $\sqrt3\sigma$ and decreasing from $\sqrt3\sigma$ onward. We distinguish four possible cases, based on the monotony of $f''$ (if $f''$ is increasing on $I_k^i$,  $\min_{\xi \in I_k^i} f''(\xi)=f''(2^{-k}i)$, $\max_{\xi \in I_k^i} f''(\xi)=f''(2^{-k}(i+1))$; if it is decreasing, $\min_{\xi \in I_k^i} f''(\xi)=f''(2^{-k}(i+1))$, $\max_{\xi \in I_k^i} f''(\xi)=f''(2^{-k}i)$; and if $\sqrt3\sigma\in I_k^i$, $\max_{\xi \in I_k^i} f''(\xi)=f''(\sqrt3\sigma)$):
\begin{itemize}
\item \textbf{Case 1:} when $2^{-k}(2i+2)\leqslant\sqrt3\sigma$, then
\begin{align*}
\frac{\left(2^{-k}\right)^3}{12}f''(2^{-k}(2i)) &\leqslant v_{(k,2i)}\leqslant\frac{\left(2^{-k}\right)^3}{12}f''(2^{-k}(2i+1))\\
-\frac{\left(2^{-k}\right)^3}{12}f''(2^{-k}(2i+2)) &\leqslant -v_{(k,2i+1)}\leqslant-\frac{\left(2^{-k}\right)^3}{12}f''(2^{-k}(2i+1))
\end{align*}

\item \textbf{Case 2:} when $2^{-k}(2i)\leqslant\sqrt3\sigma\leqslant2^{-k}(2i+1)$ then
\begin{align*}
\frac{\left(2^{-k}\right)^3}{12}f''(2^{-k}(2i)) &\leqslant v_{(k,2i)}\leqslant\frac{\left(2^{-k}\right)^3}{12}f''(\sqrt3\sigma)\\
-\frac{\left(2^{-k}\right)^3}{12}f''(2^{-k}(2i+1)) &\leqslant -v_{(k,2i+1)}\leqslant-\frac{\left(2^{-k}\right)^3}{12}f''(2^{-k}(2i+2))
\end{align*}

\item \textbf{Case 3:} when $2^{-k}(2i+1)\leqslant\sqrt3\sigma\leqslant2^{-k}(2i+2)$ then
\begin{align*}
\frac{\left(2^{-k}\right)^3}{12}f''(2^{-k}(2i)) &\leqslant v_{(k,2i)}\leqslant\frac{\left(2^{-k}\right)^3}{12}f''(2^{-k}(2i+1))\\
-\frac{\left(2^{-k}\right)^3}{12}f''(\sqrt3\sigma) &\leqslant -v_{(k,2i+1)}\leqslant-\frac{\left(2^{-k}\right)^3}{12}f''(2^{-k}(2i+1))
\end{align*}

\item \textbf{Case 4:} when $2^{-k}(2i)\geqslant\sqrt3\sigma$, then
\begin{align*}
\frac{\left(2^{-k}\right)^3}{12}f''(2^{-k}(2i+1)) &\leqslant v_{(k,2i)}\leqslant\frac{\left(2^{-k}\right)^3}{12}f''(2^{-k}(2i))\\
-\frac{\left(2^{-k}\right)^3}{12}f''(2^{-k}(2i+1)) &\leqslant -v_{(k,2i+1)}\leqslant-\frac{\left(2^{-k}\right)^3}{12}f''(2^{-k}(2i+2))
\end{align*}

\end{itemize}

The bounds on these error terms also cancel two by two, except the first one, and the two terms around $\sqrt3\sigma$ (we consider cases with $2^{-k}<\sqrt3\sigma\Leftrightarrow \frac{2^{-k}}{\sigma}<\sqrt3$, so that the maximum of $f''$ is not in the first stripe). At worst, $-\frac{\left(2^{-k}\right)^3}{12}f''(\sqrt3\sigma)$ and $-\frac{\left(2^{-k}\right)^3}{12}f''(2^{-k}(2i+1))$ with $2i+1=\lceil\sqrt3\sigma\rceil$  remain on the right side, and  $\frac{\left(2^{-k}\right)^3}{12}f''(\sqrt3\sigma)$ and $\frac{\left(2^{-k}\right)^3}{12}f''(2^{-k}(2i))$ with $2i=\lceil\sqrt3\sigma\rceil$  on the left (these worst cases cannot all occur at once).

All in all, we are left with (using that $f''$ reaches its maximum at $\sqrt3\sigma$):

\begin{align}
\frac{\left(2^{-k}\right)^3}{12}f''(0)-2\frac{\left(2^{-k}\right)^3}{12}f''(\sqrt3\sigma) &\leqslant \sum_{i\geqslant 0}(-1)^i w_{(k,i)}\leqslant 2\frac{\left(2^{-k}\right)^3}{12}f''(\sqrt3\sigma) \nonumber \\
-\frac{\left(2^{-k}\right)^3}{12\sigma^3\sqrt{2\pi}}-2\frac{\left(2^{-k}\right)^3}{12}\frac{2}{\sigma^3\sqrt{2\pi}}e^{-3/2} &\leqslant \sum_{i\geqslant 0}(-1)^i w_{(k,i)}\leqslant 2\frac{\left(2^{-k}\right)^3}{12}\frac{2}{\sigma^3\sqrt{2\pi}}e^{-3/2} \nonumber \end{align}

Finally:
\begin{align}
\frac{1}{2}+\frac{2^{-k}}{2\sigma\sqrt{2\pi}} -\frac{\left(2^{-k}\right)^3(4e^{-3/2}+1)}{12\sigma^3\sqrt{2\pi}} &\leqslant p_k\leqslant \frac{1}{2}+\frac{2^{-k}}{2\sigma\sqrt{2\pi}}  +\frac{\left(2^{-k}\right)^3(4e^{-3/2})}{12\sigma^3\sqrt{2\pi}} \label{eq:encadrement}
\end{align}

It follows that, when $\frac{2^{-k}}{\sigma}$ is small, bit $k$ contributes to the result accuracy with probability 
\begin{align}
p_k\sim \frac{2^{-k}}{2\sigma\sqrt{2\pi}} + \frac{1}{2}. \label{eq:pkeq}
\end{align}
\end{proof}

\begin{remark}
Equation \ref{eq:encadrement} shows that $p_k$ is closer to $\frac{2^{-k}}{2\sigma\sqrt{2\pi}} + \frac{1}{2}$ than to $\frac{\left(2^{-k}\right)^3(4e^{-3/2}+1)}{12\sigma^3\sqrt{2\pi}}>\frac{\left(2^{-k}\right)^3(4e^{-3/2})}{12\sigma^3\sqrt{2\pi}}$. Now, $x\mapsto \frac{x^3(4e^{-\frac32})}{12\sqrt{2\pi}}$ is increasing, and for $x=\frac12$, less than 0.05.
Hence, for $\frac{2^{-k}}{\sigma}\leqslant\frac12$, 
\begin{align*}
 \frac{2^{-k}}{2\sigma\sqrt{2\pi}} + \frac{1}{2}-0.05< p_k< \frac{2^{-k}}{2\sigma\sqrt{2\pi}} + \frac{1}{2}+0.05. 
\end{align*}
\end{remark}

If we wish to keep only bits improving the result with a probability greater than $p$, then we will keep $c$ contributing bits, with
\begin{align}\label{eq:pkeq2}
c=-\log_2(\sigma)-\log_2\left(p-\frac{1}{2}\right)-\log_2\left(2\sqrt{2\pi}\right). 
\end{align} 

As above, this formula can be further refined by replacing $\sigma$ with $\hat\sigma$ and adding a term taking into account the confidence level:
\begin{align*}
  c\geqslant
  \underbrace{-\log_2(\hat\sigma)
  -\left[\frac12\log_2\left(\frac{n-1}{\chi^2_{1-\alpha/2}}\right)
  +\log_2\left(p-\frac{1}{2}\right)
  +\log_2\left(2\sqrt{2\pi}\right)\right]}_{\hat c_{\text{\sc cnh}}}.
\end{align*}

\begin{figure}
\begin{center}\includegraphics[width=.6\linewidth]{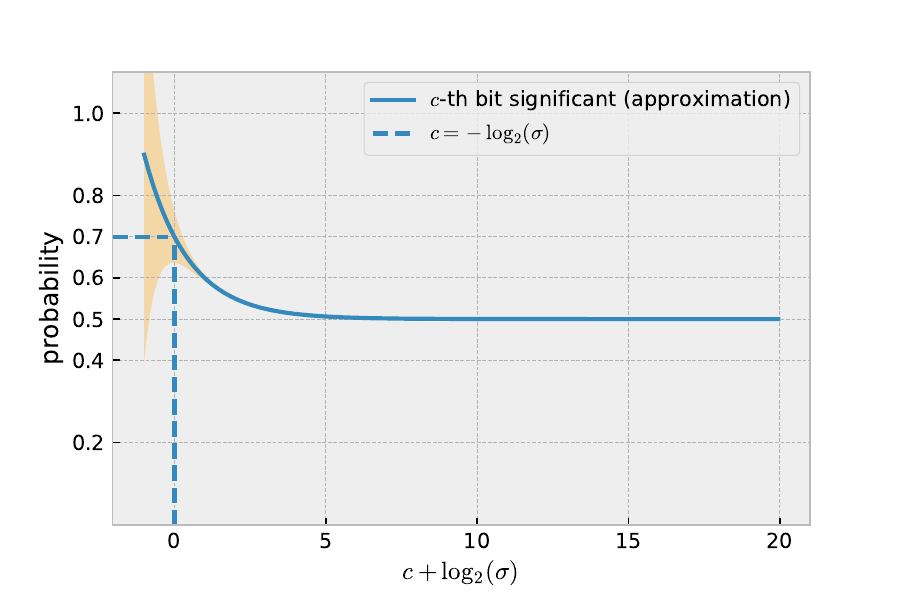}\end{center}
\caption{Profile of the contribution bit curve: when the dashed line is positioned on the $-\log_2\sigma$ abscissa, the curve corresponds to the approximation~\ref{eq:pkeq} of the probability that the bit contributes to the result accuracy. The shaded area represents the bound on the error given by equation~\ref{eq:encadrement}. 
}
\label{fig:approx}
\end{figure}

Figure~\ref{fig:approx} plots the approximation of equation~\ref{eq:pkeq}.
We note that for a centered normal distribution the probability of contribution decreases monotonically  towards $0.5$. Close to $0.5$, bits become more and more indistinguishable from random noise since their 
probability is not affected by the computation.

The approximation of equation~\ref{eq:pkeq} is tight for $k > -\log_2\sigma$: in this case, the absolute error of the approximation formula is less than $2~\%$. 
The probability of contribution at 
$k=-\log_2\sigma$ is $0.7$. Therefore, equation~\ref{eq:pkeq2} can be safely used for probabilities less than $0.7$. In this paper, we want to find the limit after which bits are random noise. This limit corresponds to a probability of $0.5$ and the approximation is tight for $p<0.7$.

\paragraph{Application}

\begin{figure}
\centering
\includegraphics[width=.48\linewidth]{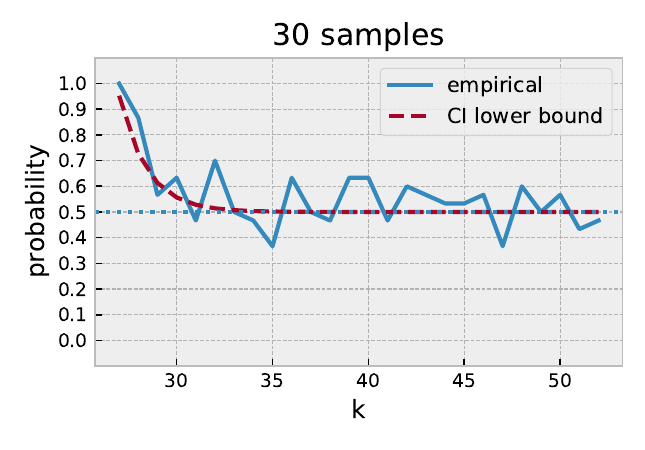}
\includegraphics[width=.48\linewidth]{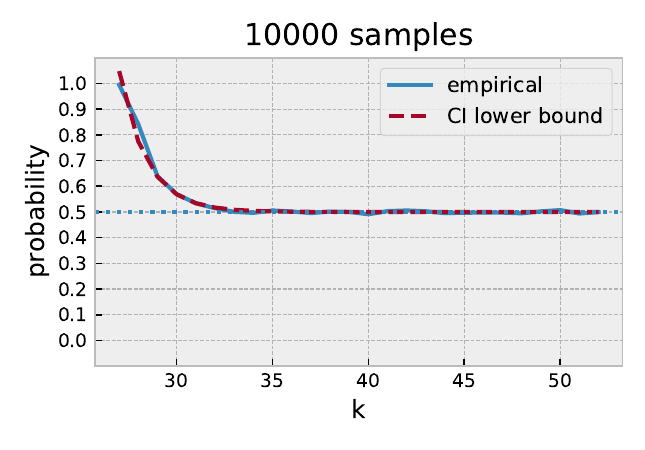}\caption{Contributing bits for Cramer $x[0]$ variable computed under the normal hypothesis using 30 and 10000 samples with the approximation of equation~\ref{eq:pkeq}.}
\label{fig:cramer-0-cb-predicted}
\end{figure}

Figure~\ref{fig:cramer-0-cb-predicted} shows that the approximation proposed in
this section tightly estimates the empirical samples in Cramer $x[0]$ example.

If we consider a 51~\% threshold for the contribution of the bits we wish to keep, then we should keep $c=-\log_2(\sigma)-\log_2(p-\frac{1}{2})-\log_2(2\sqrt{2\pi})=-\log_2(\sigma) + 4.318108$. As in section~\ref{sec:significant}, we estimate $-\log_2(\sigma)$ with a 95~\% Chi-2 confidence interval, and compute $c=32.8$.

This means that with probability 51\% the first 32 bits of the mantissa will round the result towards the correct reference value. After the 34th bit the chances of rounding correctly or incorrectly are even: the noise after the 34th bit is random and does not depend on the computation. Bits 34 onwards can be discarded.

\subsection{Summary of results for a  Normal Centered Distribution}
\label{sec:normal-summary}

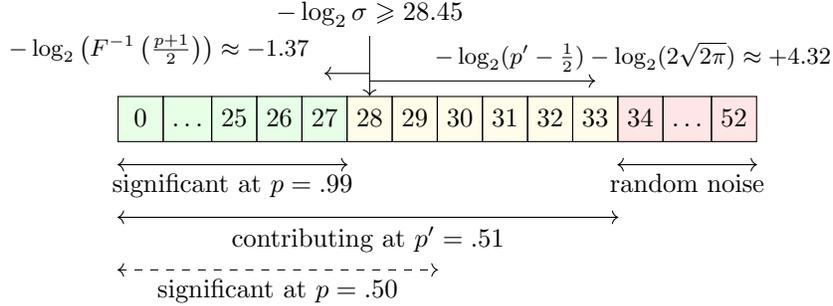
\begin{figure}
\begin{center}
\begin{tikzpicture}[
array/.style={matrix of nodes,nodes={draw, minimum size=6mm, text height=1.5ex,  text depth=.25ex, anchor=center}, column sep=-\pgflinewidth, row sep=.5mm, 
nodes in empty cells}]

\matrix[array] (array) {
  0 & $\ldots$ & 25 & 26 & 27 & 28 & 29 & 30 & 31 & 32 & 33 & 34 & $\ldots$ & 52\\
};

\begin{scope}[on background layer]
\fill[green!10] (array-1-1.north west) rectangle (array-1-5.south east);
\fill[yellow!10] (array-1-6.north west) rectangle (array-1-11.south east);
\fill[red!10] (array-1-12.north west) rectangle (array-1-14.south east);
\end{scope}

\draw[<-] (array-1-6.north)--++(90:8mm) node [above] (first) {$-\log_2\sigma \geqslant 28.45$};

\draw[->]([yshift=2mm]array-1-6.north) -- node[above, font=\small, xshift=20mm] {$-\log_2(p'-\frac{1}{2})-\log_2(2\sqrt{2\pi}) \approx +4.32$} ([yshift=2mm]array-1-11.north);

\draw[->]([yshift=3mm]array-1-6.north) -- node[above, font=\small, xshift=-25mm] {$-\log_2 \left(F^{-1}\left(\frac{p+1}{2}\right)\right) \approx -1.37$} ([yshift=3mm]array-1-5.north);

\draw[<->]([yshift=-3mm]array-1-1.south west) -- node[below] {significant at $p=.99$} ([yshift=-3mm]array-1-5.south east);

\draw[<->]([yshift=-10mm]array-1-1.south west) -- node[below] {contributing at $p'=.51$} ([yshift=-10mm]array-1-11.south east);
\draw[dashed, <->]([yshift=-17mm]array-1-1.south west) -- node[below] {significant at $p=.50$} ([yshift=-17mm]array-1-7.south east);
\draw[<->]([yshift=-3mm]array-1-12.south west) -- node[below] {random noise} ([yshift=-3mm]array-1-14.south east);

\end{tikzpicture}
\end{center}
\caption{Summary of results on the $X[0]$ Cramer example.  $28.45 \approx -\log_2\hat\sigma - \frac{1}{2}\log_2\left(\frac{n-1}{\chi_{1-\alpha/2}^2}\right)$}
\label{fig:mantissaexample}
\end{figure}

Under the normality hypothesis, the quantity $-\log_2\frac{\sigma_X}{|\mu_X|}$ introduced by Stott Parker is pivotal, but needs to be refined. In our framework, Stott Parker's definition maps to $Z=\frac{X}{|\mu_X|}-1$, which computes the relative error to the mean. In this case Stott Parker's formula computes the position of the bit until which the result has 68~\% chance of being significant, and that contributes to the result precision with a probability of 70~\%.

However, from this bit, for each desired probability level, there is a simple way to compute a quantity by which to move back to be sure that the result is significant. It is also easy to compute a quantity by which to move forward in order to guarantee that all bits contributing more than a fixed level are kept.
Figure~\ref{fig:mantissaexample} demonstrates this on the Cramer's example. 

We recall from section~\ref{sec:example} that in this example $\hat s_{\text{mca}} = -\log_2\left|\frac{\hat\sigma_X}{\hat\mu_X}\right| = -\log_2(\hat\sigma) \approx 28.48$. First a lower bound, 28.45, on $-\log_2 \sigma$ is computed with the Chi-2 95~\% confidence interval. With this confidence level, it is a lower bound of $s_{\text{mca}}$ (as introduced in definition~\ref{def:significant.mca}). It is also a lower bound of $s_{\text{sto}}$ with 68~\% probability (as introduced in definition~\ref{def:significant}).  To compute a lower bound on bits that are significant with probability 99~\%, we simply subtract 1.37 from this number. By adding 4.32 to this number we get the number of bits that contribute or round towards the reference with a probability of 51~\%. The remaining bits in the mantissa are random noise.

It is important to understand the difference between contributing and significant bits. To illustrate this difference, we show in figure~\ref{fig:mantissaexample} the number of significant bits with 50~\% probability which we estimate at
29 bits (28.45 shifted by +0.57 bits). We deduce, since the probability of significant bits decreases monotonically, that bits in the range 30-33 are significant with a probability under 50~\%; in other words they are likely to be non significant. Yet \emph{taken individually}, these bits are contributing with probability over 51\%. Therefore, bits in the range 30-33 still contain useful information about the computation and cannot be considered random noise. It is up to the practitioner to decide how many bits to keep depending on their use-case. 

Taking this into account, we propose to give a result, by printing all contributing bits at the chosen probability and confidence levels and an annotation bounding the error term at the chosen probability and confidence levels.
This would result, for $k=-\log_2\sigma$, in the following: for an absolute error with $\lceil k+(e_y-1)+4.318108\rceil$ bits with the annotation $\pm 2^{\lfloor k+(e_y-1) - 1.365037\rfloor}$ at 99~\%;
for a relative error with $\lceil k+4.318108\rceil$ and $\pm 2^{\lfloor k-1.365037\rfloor} \times y$ at 99~\% for a relative error. In this notation, only digits that are likely to round correctly the final result with a probability greater than 1~\% are written; the error at probability 99~\% is written. In decimal, this notation takes up to two additional digits ($4.318108\times \log_{10}2\approx 1.29$~digits) that are probably wrong, but still have a chance to contribute to the result precision. As an example, using this notation to display the IEEE-754 result of Cramer's $X[0]$ yields, with 9 contributing digits and 8 significant digits:
\begin{center}
1.999999996 $\pm$ 1.4e-08 (at 99\% with confidence 95\%).
\end{center}
These 10~digits contain all the valuable information in the result, and are the only ones that it would make sense to save, for example in a checkpoint-restart scheme. 

\bigskip

Interestingly enough, the CESTAC definition of significance can be reinterpreted in this statistical framework.
Equation~\eqref{eq:scestac} defines the CESTAC estimator as
\begin{align*}
 \hat s_{\text{\sc cestac}} &= \log_2\left(\frac{\tau_n \, \hat \sigma_X}{\sqrt{n} \; |\hat \mu_X|}\right) \\
                            &= \log_2\hat \sigma - \log_2\frac{\tau_n}{\sqrt n} \quad \textrm{ (taking } Z = \frac{X}{|\mu_X|} -1 \textrm{)}
\end{align*}
This estimator was originally designed to estimate $s_{\text{\sc cestac}}$ (as introduced in definition~\ref{def:significant.cestac}) under the CESTAC model hypothesis. We showed previously that the formula tends to infinity when increasing the number of samples $n$. Yet CADNA~\cite{lamotte2010cadna_c}, the most popular library implementing CESTAC, sets $n=3$ and $1-\alpha=95~\%$. In this case,
\begin{align*}
\hat s_{\text{\sc cadna}} &\approx \log_2\hat \sigma - 1.31.
\end{align*}

Reinterpreting the -1.31 shift as the $\delta_{\text{\sc cnh}}$ term from equation~\eqref{eq:scnh}, we see that $\hat s_{\text{\sc cadna}}$ can be seen as an estimator for our stochastic definition of significant bits, $s_{\text{sto}}$, with probability 30.8\% at a 95\% confidence level. 
With only three samples, using $\hat\sigma$ as an estimator of $\sigma$ can introduce a strong error. The term $\frac12\log_2\left((n-1)/\chi^2_{1-\alpha/2}\right)$ which accounts for this error inside $\delta_{\text{\sc cnh}}$ becomes important, which explains the low probability (30.8\%) of the estimation. 

To mitigate this issue, it is recommended to take a safety margin of 1 decimal digit from the number of significant digits estimated by CADNA. In our formalism, shifting $\hat s_{\text{\sc cadna}}$ further by 1 decimal digit (or approximately 3.32 bits), the result can be
reinterpreted as a shift of $-1.31-3.32=-4.63$~bits, estimating $s_{\text{sto}}$ with probability over 99\% (with 95\% confidence).

\section{Accuracy in the General Case}
\label{sec:general}

The hypothesis that the distribution $Z$ is normal, or that it has expectation $0$ is not always true. We propose statistical tools to study the significance of bits as well as their contribution to the result accuracy that do not rely on any assumption regarding the distribution of the results.

To address the problem in the general case we reframe it in the context of Bernoulli estimation, which is interesting because:
\begin{itemize}
  \item it does not rely on any assumptions on the distribution of $Z$;

   \item it provides a strong confidence interval for determining the number
  of significant digits when using stochastic arithmetic methods;
  
  \item thanks to a more conservative bound, it allows to estimate {\it a priori} in all cases for a given probability and confidence a safe number of sample to draw from the Monte Carlo experiment.
  
\end{itemize}

\subsection{Background on Bernoulli estimation}

In the next section, we  restate the problem of estimating the number of significant bits as a series of estimations of Bernoulli parameters. We present here some basic results on such estimations.

Consider a sequence of independent identically distributed Bernoulli experiments with an unknown parameter $p$ and outcomes $(p_i)$. Each value of the parameter $p$ gives a model of this experiment, and, among them, we will only keep an interval of model parameters under which the probability of the given observation is greater than $\alpha$. The set of possible values for $p$ will then be called a confidence interval of level $1-\alpha$ for $p$ :
if the actual value of $p$ is not in the confidence interval computed from the outcome, it means that the observed outcome was an ``accident'' the probability of which is less than $\alpha$.

A case of particular interest in our study is the one when all experiments succeed. Then, the probability of this outcome is $p^n$ under the model that the Bernoulli parameter value is $p$. We then reject models (i.e., values of $p$) such that $p^n<\alpha\Leftrightarrow n\ln(p)<\ln(\alpha)$. Now, $\ln(p)\leqslant p-1$ and $\ln(p)\sim p-1$ is a first order approximation when $p$ is close to 1. Thus, taking $p<1+\frac{\ln(\alpha)}n$ leads to a probability of the observation less than $\alpha$, and one can reject these values of $p$. In particular, taking $1-\alpha=95~\%$, we keep values of $p$ greater than $1-\frac3n$, and $\left[1-\frac3n, 1\right]$ is a $95~\%$ confidence interval. This result is known in clinical trial's literature as the \emph{Rule of Three}~\cite{Eypasch619}. Vice versa, in an experiment with no negative outcome, one can conclude with confidence $1-\alpha$ that the probability of a positive outcome is greater than $p$ after $\left\lceil\frac{\ln(\alpha)}{\ln(p)}\right\rceil$ positive trials.

The general case can be dealt with by using the Central Limit Theorem, which shows that for a number $n$ of experiments large enough (with respect to $\widehat p=\frac1n\sum p_i$), $\sqrt{n}(\widehat p-p)/(\widehat p(1-\widehat p))$ is close to a Gaussian random variable with law $\mathcal N\left(0, 1\right)$. This approximation is known to be unfit in many cases, and can be improved by considering $\tilde{p}=\frac1{n+4}(\sum p_i+2)$ rather than $\hat p$ as shown by Brown et al.~\cite{brown2001interval} (this paper also presents other estimators to build confidence intervals in this situation; in particular, it proposes a revised method when $\tilde p$ is close to 0 or 1, a situation in which the confidence interval below may be overly optimistic). 
Then, with $F$ the cumulative distribution function of $\mathcal N(0, 1)$, $$\left[\tilde p-\sqrt{\tilde p(1-\tilde p)/n}F^{-1}(1-\alpha/2), \tilde p+\sqrt{\tilde p(1-\tilde p)/n}F^{-1}(1-\alpha/2)\right]$$ is an $1-\alpha$ confidence interval for $p$. If we focus on a lower bound on the parameter $p$, we can also use $\left[\tilde p-\sqrt{\tilde p(1-\tilde p)/n}F^{-1}(1-\alpha), 1\right]$ as a confidence interval of level $1-\alpha$.

Thus, from $n$ independent experiments, of which $n_s$ have been a success, we can affirm with confidence 95~\% that the probability of success is greater than $\frac{n_s+2}{n+4}-1.65\sqrt{\frac{(n_s+2)(n-n_s+2)}{(n+4)^3}}$. We can note that when $n_s=n$, this confidence interval is valid, but much more conservative than the one obtained above, that can thus be preferred in this particular case.

\subsection{Statistical formulation as Bernoulli trials}

Now, for each of the four discussed settings, presented in section~\ref{sec:problem}, we can form two series of Bernoulli trials based on collected data. 

\begin{center}
\begin{tabular}{lll}
\toprule
				  &reference $x$ & reference $Y$ \\
absolute precision& $Z=X-x$       & $Z=X-Y$      \\
relative precision& $Z=X/x-1$     & $Z=X/Y-1$    \\
\bottomrule
\end{tabular}
\end{center}

When the reference is a constant $x$, we consider $n$ samples $X_i$. We form $N$ pieces of data by computing $Z_i=X_i-x$ or $Z_i=X_i/x-1$ respectively. 

When the reference is another random variable $Y$, we form $N$ pieces of data by computing $Z_i=X_i-Y_i$ or $Z_i=X_i/Y_i-1$.
In the case where $X=Y$ and we study the distance between samples of a random process, this requires $2N$ samples from $X$.

From these $N$ pieces of data, we form Bernoulli trials by counting the number of success of 
\begin{align*}
  S_i^k= \indicator_{|Z_i|<2^{-k}}
\end{align*}
for studying $k$-th bit significance, and 
\begin{align*}
C_i^k=\indicator_{\lfloor 2^{k}\left|Z_i\right| \rfloor \textrm{ is even}}
\end{align*} 
for studying $k$-th bit contribution, where $\indicator$ is the indicator function.

From these two Bernoulli samples, the estimation can be made as above to determine the probability that the $k$-th bit is significant and the probability that it contributes to the result, for any $k$. The result can then be plotted as two probability plots, one for significance, the other for the contribution. The significance plot is non-increasing by construction, should start at 1 if at least one bit can be trusted, and tends to 0. The contribution plot should tend to $\frac{1}{2}$ in most cases, since the last digits are pure noise and are not affected by the computation.

\subsection{Evaluation}

The main goal of the Bernoulli formulation is to deal with non normal distributions.
In this section, we evaluate the Bernoulli estimate on Cramer's $X[0]$ samples which follow a normal distribution. This is to keep a consistent example across the whole paper and to compare the results with the Normal formulation estimates. Later, in section~\ref{sec:experiments}, we will apply the Bernoulli estimate to distributions produced by the industrial simulation codes EuroPlexus and code\_aster, some of which are not normal.

Figure~\ref{fig:accuracy-cramer-0} plots the significance and the contribution per bit probabilities for $X[0]$ using the Bernoulli estimation. The estimation closely matches the empirical results.
It is interesting to compare the Bernoulli estimates with 30 samples to the Normal estimates in figures~\ref{fig:cramer-0-sb-predicted} and~\ref{fig:cramer-0-cb-predicted}. The Bernoulli estimates are less tight and more conservative. This is expected since they do not build upon the normality assumption of the distribution.

\begin{figure}
\centering
\begin{subfigure}{.48\linewidth}
\includegraphics[width=\linewidth]{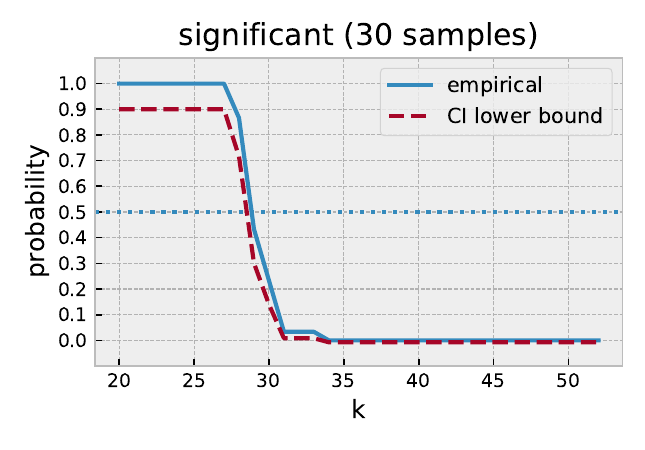}
\includegraphics[width=\linewidth]{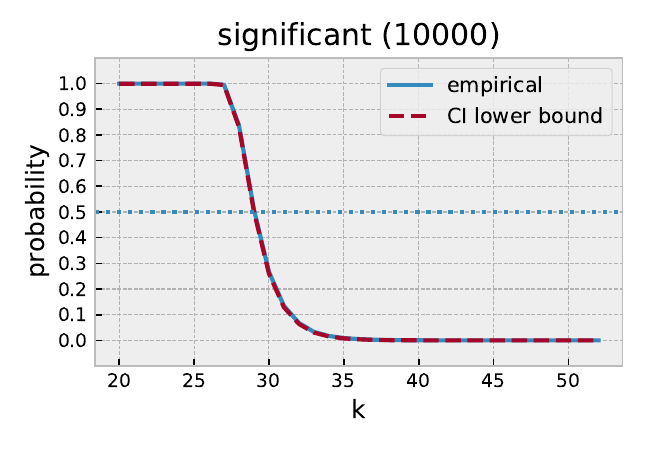}
\end{subfigure}\hfill
\begin{subfigure}{.48\linewidth}
\includegraphics[width=\linewidth]{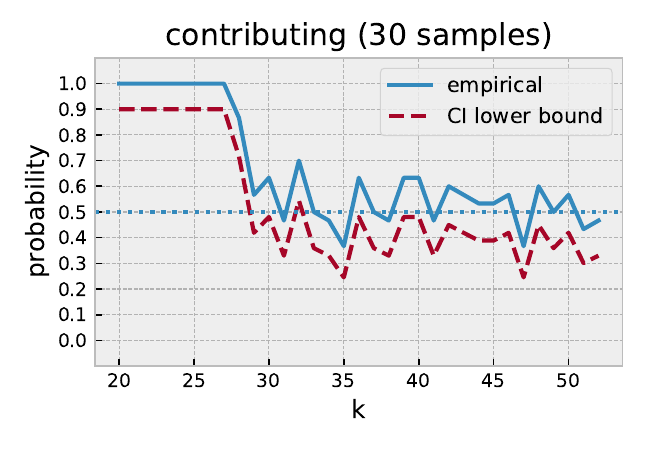}
\includegraphics[width=\linewidth]{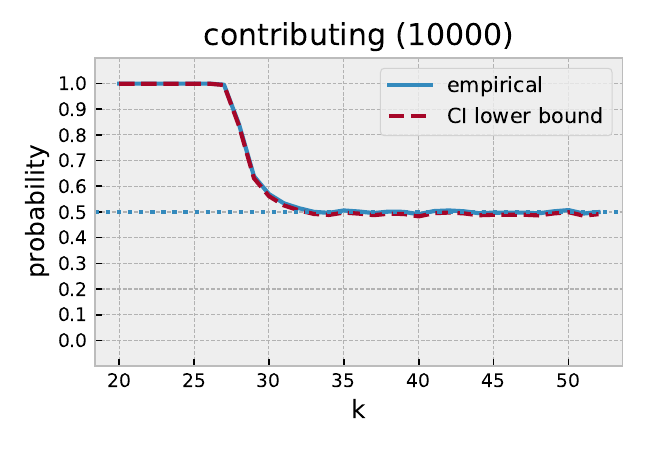}
\end{subfigure}
\caption{Significance and contribution per bit for variable $X[0]$ of the Cramer's system with 30 and 10000 samples.}
\label{fig:accuracy-cramer-0}
\end{figure}

If we are only interested in the number of significant digits, we can consider the Bernoulli trial with no failed outcomes since it provides an easy formula giving the required number of samples. In this case, the number of needed samples is $n = \lceil -\frac{\ln \alpha}{\ln p} \rceil$.  We then determine the maximal index~$k$ for which the first $k$~bits of all $n$~sampled results coincide with the reference:
\begin{align}
  \hat s_{\text{\sc b}} = \max \left\{ k \in \{ 1, 2, \ldots, 53\} \text{ such that } \forall i\in\left\{ 1, 2, \ldots ,n\right\}, S_i^k \text{ is true} \right\}.
  \label{eq:sb}
\end{align}

We applied this method to the $X[0]$ sample from section~\ref{sec:example}. Assuming a $(1-\alpha)=95\%$ confidence interval and a probability of $p = 99\%$  of getting $s$ significant digits, we gather $n=299$~samples. Among the collected samples, the $27$\textsuperscript{th} digit is sometimes different compared with the reference solution, but the first $26$ digits coincide for all samples. Therefore we conclude with probability $99\%$ and $95\%$ confidence, that the first $\hat s_{\text{\sc b}}=26$ binary digits are significant.

\section{Experiments on industrial use-cases}
\label{sec:experiments}

\subsection{Reproducibility analysis in the Europlexus Simulation Software}
\label{sec:epx}

In this section we show how to apply our methodology to study the numerical reproducibility of the state-of-the-art Europlexus~\cite{EPXweb} simulator. Europlexus
is a fast transient dynamic simulation software co-developed by the French
Commissariat \`a l'\'Energie Atomique et aux \'Energies Alternatives~(CEA), the
Joint Research Center~(JRC) of the European Commission, and other industrial and
academic partners. While its first lines of code date back to the mid seventies,
the current source code has grown to about 1 million lines of Fortran 77 and
Fortran 90. Europlexus runs in parallel on distributed memory architectures
through a domain decomposition strategy, and on shared memory architectures
through loop parallelism.

Europlexus has two main fields of application: simulation of severe accidents
in nuclear reactors to check the soundness of the mechanical confinement
barriers of the radioactive matters for the CEA; and simulation of explosions in
public places in order to measure their impact on the surrounding citizens and
structures for the JRC.

It handles several non-linearities, geometric or material, some of which lead to
a loss of unicity of the evolution problem considered. This is for example the
case for some configurations with frictional contact between structures, or when
the loadings cause fracture and fragmentation of the matter. Another obvious
source of bifurcations of the dynamical system is the dynamic buckling.

Due to the small errors introduced by the floating point arithmetic, the
introduction of parallel processing in Europlexus raises a difficulty for the
developer and the users: the solutions of a given simulation may differ
when changing the number of processors used for the computation.
We show here how the confidence intervals proposed in this paper help the developer to design relevant non-regression tests. To this end, we study in the following a simple case which could serve as a non-regression test, and which is symptomatic of a non-reproducibility related to FP arithmetic. It involves a vertical doubly clamped column to top and bottom plates. A vertical pressure is applied by lowering the top plate, which causes buckling of the column. The column is modeled as a set of discrete elements (here segments) connected at moving points called nodes.

\begin{figure}
  \centering
  \includegraphics[width=.75\textwidth]{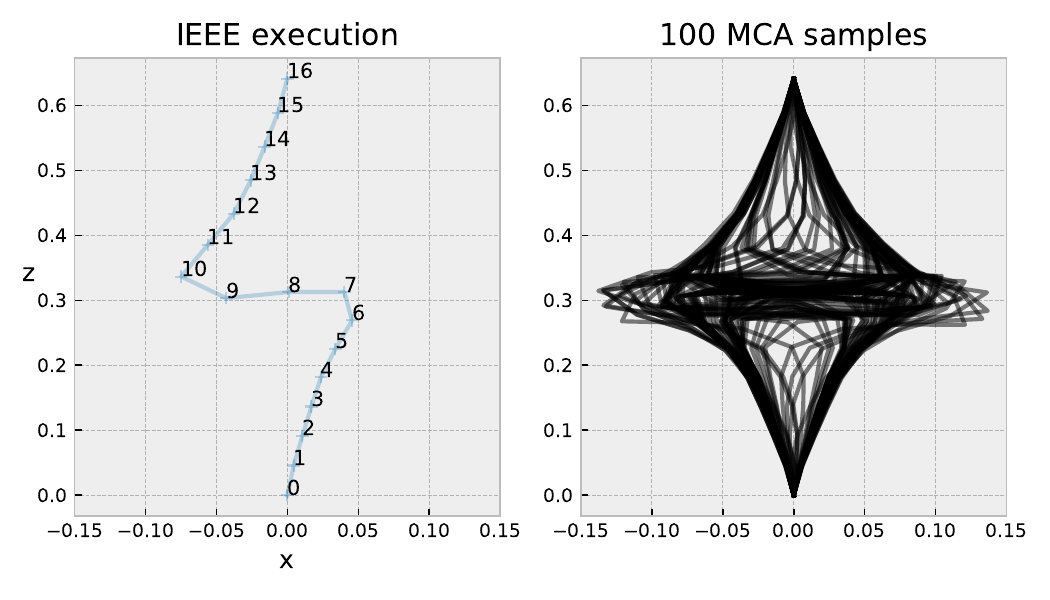}
  \caption{Europlexus buckling simulation of a doubly clamped column subject to a vertical pressure.
    The nodes of the column are labeled
    in the plot from 0 to 16.
    The left plot shows the deterministic and reproducible results produced by an IEEE-754 run of the simulation.
    In the right plot, a two-digits numerical error is simulated by collecting one hundred MCA samples with Verificarlo ($t=50$). The buckling direction is completely dominated by the small numerical error introduced.
    \label{fig:buckling}}
\end{figure}

The left plot in figure~\ref{fig:buckling} shows the result after 300 simulation time-steps with the out-of-the-box Europlexus software using standard IEEE arithmetic. The sequential result is deterministic and does not change when run multiple times. We wish to study how the simulation is affected by small numerical errors.

We run the same simulations but this time using the Verificarlo~\cite{verificarlo} compiler to introduce MCA randomized floating
point errors with a precision of $t=50$. The cost to instrument the whole Europlexus software and its accompanying mathematical libraries was low. In particular no change to the source-code was necessary thanks to the transparent approach to instrumentation of Verificarlo, only the build system had to be configured to use the Verificarlo compiler.

The right plot in figure~\ref{fig:buckling} shows the result of one hundred Verificarlo executions. The direction of the buckling is chaotic and completely dominated by the small FP errors introduced. This is not surprising as the buckling direction is physically unstable.

When parallelizing Europlexus or making changes to the code, it is important to
check that there are no regression on standard benchmarks. Changing the order
of the floating point operations may introduce small rounding errors.  As we
just saw, even small numerical errors change the buckling direction. This makes such a benchmark unsuitable for classical non-regression tests.

The column is modeled as a set of discrete elements connected by nodes. The distribution on the $x$-axis is normal but whatever the node there are no significant digits among the samples. The variation between samples is strong on the $x$-axis, the $x$ position clearly cannot be used as a regression test.

\begin{figure}
\centering
\includegraphics[width=\linewidth]{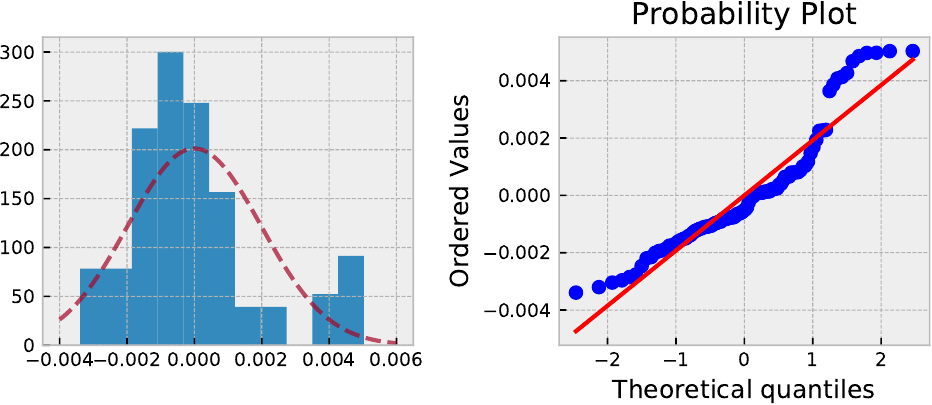}
\caption{Non normality of buckling samples on z axis and node 1. Shapiro Wilk rejects the normality hypothesis.}
\label{fig:normality-buckling}
\end{figure}

The distribution along the $z$ axis is more interesting as it is non normal for all the nodes. Figure~\ref{fig:normality-buckling} shows the quantile-quantile plot for node 1 (Shapiro-Wilk rejects normality with $W = 0.9$ and $p = 1.8e-06$). Because the distribution on the $z$ axis is non normal we should apply the Bernoulli significant bits estimator. In this study, we measure the number of significant digits considering the relative error against the sample mean, so $Z=\frac{X}{\mu_X} - 1$.

\begin{figure}
\centering
\includegraphics[width=.6\textwidth]{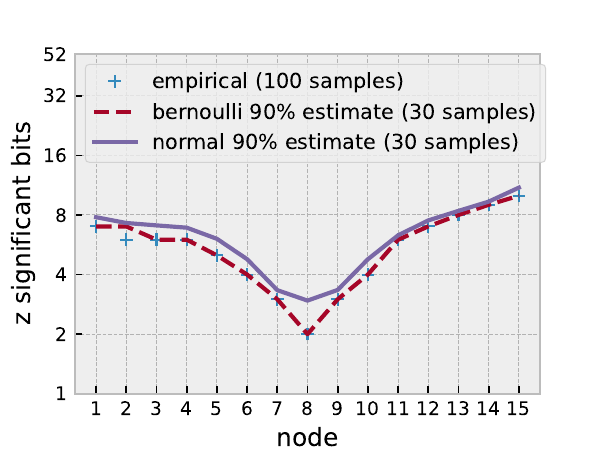}
\caption{Significant bits on the $z$ axis distribution. Bernoulli estimation
captures precisely the behavior (except for node 2). Normal formula overestimates the number of digits, this is expected since the distribution is strongly non normal.
}
\label{fig:sbits-buckling}
\end{figure}

\begin{figure}
\centering
\includegraphics[width=\textwidth]{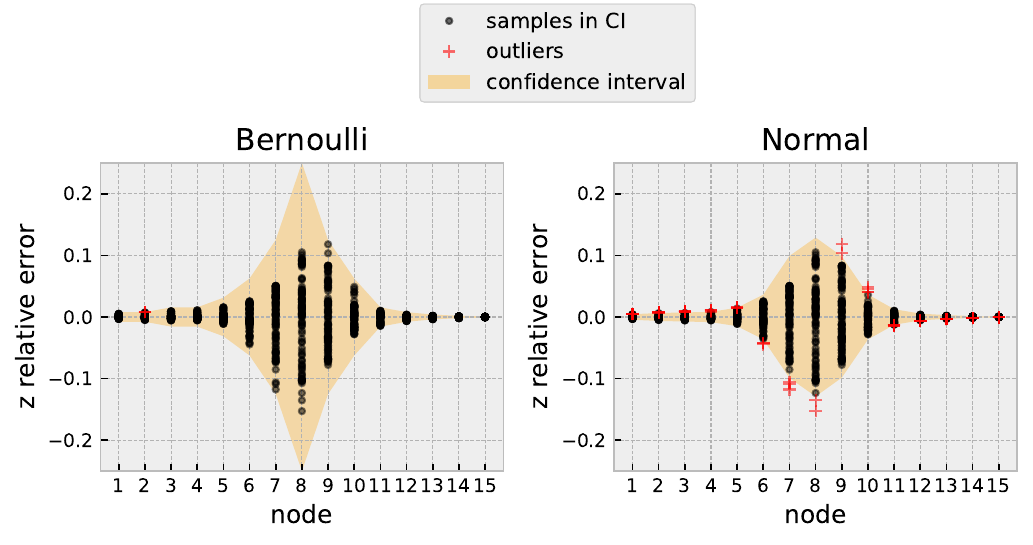}
\caption{Relative error between the samples and the mean of the $z$-axis distribution. The shaded envelope corresponds to the computed confidence interval with 30 samples. Black dots are samples that fall inside the CI. Red crosses are outliers that fall outside the CI. In the Bernoulli case, only 3 samples out of 70 fall outside of the interval; which is compatible with the 90\% probability threshold.}
\label{fig:confidence-buckling}
\end{figure}

To test the robustness of the proposed confidence interval, we
computed the Bernoulli estimate on the first 30 samples of 
the distribution. This corresponds to a probability of 90\% with a confidence of 95\%. We also computed the Normal estimate on the first 30 samples with the same probability and confidence.

Figure~\ref{fig:sbits-buckling} compares the estimates to the empirical distribution observed on 100 samples.  The Bernoulli estimate on 30 samples is fairly precise and accurately predicts the number of significant bits (except for node 2). The clamped node 16 has a fixed position and therefore all its digits are significant. The other nodes have between 2 and 10 significant digits depending on their position.

Figure~\ref{fig:confidence-buckling} shows the expected relative error on each node. We see that the Bernoulli estimate is robust and only mis-predicts the error on three samples of node 2. On the other hand, as expected, the Normal formula is not a good fit in this case due to the strong non normality of the distribution: the normal estimate is too optimistic and fails to capture the variability of the  distribution.

The previous experiments show that the $x$-axis has no significant digits and that the $z$-axis distribution has between 2 and 10 significant digits. For example node 6 has 4 significant digits on the $z$-axis. Therefore, if the practitioner uses the $z$-axis position in this benchmark as a regression test, she should expect the first four digits of the mantissa to match in 90~\% of the runs. If the error is higher than that then a numerical bug has probably been introduced in the code.

\begin{figure}
  \centering
  \includegraphics[width=.75\textwidth]{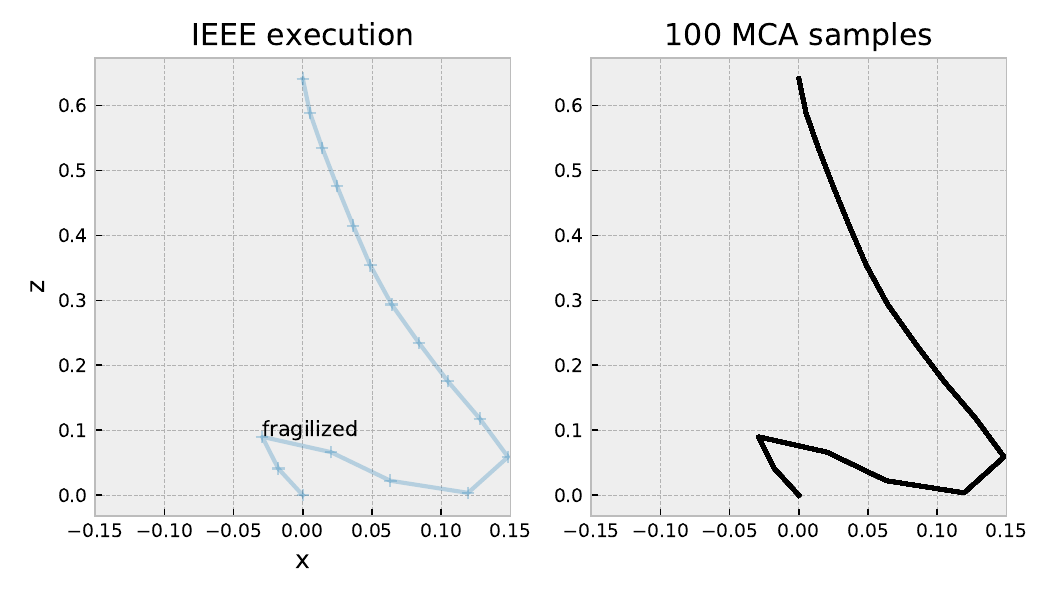}
  \caption{Europlexus buckling simulation with a fragilized node 2.
    By weakening the column, the physical process becomes reproducible in the presence of small numerical noise.
    \label{fig:buckling-fragilized}}
\end{figure}

Another possibility for the practitioner is to adapt slightly the benchmark to make it more robust to numerical noise so it can be used in regression tests.
For example, we can introduce a small perturbation in the numerical model by slightly moving node 2 along the x-axis. Then the buckling is expected to always occur in the same direction. Figure~\ref{fig:buckling-fragilized} shows what happens when node 2 is slightly displaced: the buckling becomes deterministic and robust to numerical noise: 51 bits are significant for the $x$-axis and $z$-axis samples with probability 90~\%; the two bits of precision lost correspond to the stochastic noise introduced. In this case, stochastic methods allow checking that the benchmark has become deterministic and assessing its resilience to noise.

\subsection{Verification of a numerical stability improvements in code\_aster}
\label{sec:aster}

code\_aster \cite{aster} is an open source scientific computation code dedicated to the
simulation of structural mechanics and thermo-mechanics. It has been actively
developed since 1989, mainly by the R\&D division of EDF, a
French electricity utility. It uses the finite elements method
to solve models coming from the continuum mechanics theory, and can be
used to perform simulations in a wide variety of physical fields such as mechanics, thermal physics, acoustics, seismology and others.
code\_aster has a very large source
code base, with more than 1,500,000 lines of code. It also uses numerous
third-party software, such as linear solvers or mesh manipulation tools. Its development team has been dedicated to code quality for a long time, and has
accumulated several hundreds of test cases which are run daily as part of
the Verification \& Validation process.
\smallskip

In a previous study~\cite{fevotte2017studying} of code\_aster, the Verrou tool was used to assess the numerical quality of quantities tested in the non-regression
database. The error localization features of Verrou were used to find the origin of errors in the source code, and improvements were proposed. In
this section, we first summarize the results of this study, before
using the new estimators described in this paper to confidently assess
the benefits of the proposed corrections.

The study focuses on one test case named \texttt{sdnl112a} \cite{sdnl112}, which
computes 6~quantities related to the vibrations of steam generator tubes in
nuclear reactors. These quantities will be denoted here as $a$, $b$, $c$, $d$,
$e$ and $f$. In the original implementation of code\_aster -- which will be
referred to as \texttt{version0} in the following -- the test case successfully
completes on some hardware architectures, and fails on others. In this case,
failing means that some quantities are computed with relative discrepancies
larger than $10^{-6}$ with respect to some reference values. Using the Verrou
tool to assess the numerical quality of these results fails: some runs
perturbated by random rounding bail out without producing results. This exhibits
somewhat severe instabilities, but does not allow quantifying their impact on
the results accuracy.

The methodology then proceeds to finding the origin of such instabilities using
Verrou. In a first stage, a test coverage comparison between two samples
uncovers an unstable branch which is reproduced, after simplification, in figure
\ref{aster:fun1:org}. A reformulation of the incriminated computation
(cf. figure \ref{aster:fun1:cor}) is proposed, which leads to a first corrected
implementation, which will be referred to as \texttt{version1}. This version
still does not pass automated non-regression testing. However, this time,
assessing the stability of the results using Verrou and the standard MCA
estimator (with 6 samples) yields meaningful results. Quantity $e$ is evaluated
to be the most problematic, with only 5 reliable significant decimal digits
(19~bits), when users expect at least 6. Other computed quantities seem to meet
the expected precision. For example, 9~decimal digits (30~bits) are estimated to
be reliably computed for quantity $a$. In the following, we will use notations
consistent with the rest of this paper, the computed result~$x$ being either
quantity $a$ or $e$.

\begin{figure}
\ifpreprint
  \hspace{-1cm}
\fi
  \hfill
  \lstset{language=fortran}
\ifpreprint
  \begin{minipage}[t]{0.4\linewidth}
\else
  \begin{minipage}[t]{0.33\linewidth}
\fi
    \scriptsize
  \begin{lstlisting}
if (a.eq.b) then
  area = a
else
  area = (b-a) / (log(b)-log(a))
endif
  \end{lstlisting}
\caption{Unstable branch detected in code\_aster}
  \label{aster:fun1:org}
\end{minipage}
\hfill\hfill
  \begin{minipage}[t]{0.47\linewidth}
    \scriptsize
  \begin{lstlisting}
if (abs(a-b).lt.tol* min(abs(a),abs(b))) then
  area = a
else
  r = a / b
  area = a * (r-1) / log(r)
endif\end{lstlisting}
  \caption{Unstable branch corrected}
  \label{aster:fun1:cor}
\end{minipage}
\hfill\strut
\end{figure}

In a second stage, the delta-debugging feature of Verrou is used to locate the origin of remaining instabilities in \texttt{version1}. The delta-debugging pinpoints the dot product
in a sparse matrix vector product
(routine \texttt{mrmmvr}). A first approach to  mitigate the loss of
precision, implemented in \texttt{version2}, consists in using a compensated
summation algorithm in the dot product: algorithm \texttt{sum2} from~\cite{ogita2005}. In \texttt{version3},
a fully compensated dot product algorithm is implemented: algorithm \texttt{dot2}
from~\cite{ogita2005}.

Still using the standard MCA estimator, the study concludes that
\texttt{version2} slightly improves results precision: respectively 30 and
19~reliable bits for quantities $a$~and~$e$. Only \texttt{version3} meets user
thresholds, with respectively 32 and 21~reliable bits for $a$~and~$e$. This
version is also the only one to pass automated non-regression and verification
tests.

These different variants of code\_aster, along with the estimated quality of
their results, are summarized in table~\ref{MCATab}.

\begin{table}
\begin{tabular}{lcccl}
\toprule
Implementation
    & \multicolumn{2}{c}{$\hat s_{\text{\sc mca}}$}
    && comment\\ 
 &$a$& $e$ & & \\ 
\midrule 
 \texttt{version0}  &   Fail  & Fail  & & original version\\
 \texttt{version1}  &   30.89 & 19.73 & & fixes an unstable test\\
 \texttt{version2}  &   30.96 & 19.80 & & compensated summation\\
 \texttt{version3}  &   32.82 & 21.65 & & fully compensated dot product\\
  \bottomrule
\end{tabular}

\caption{Summary of the numerical quality assessment of 4 versions of
    code\_aster, using Verrou and the standard MCA estimator with 6 samples. 
}
\label{MCATab}
\end{table}
 
\medskip

Therefore, \texttt{version3} should be considered as a good candidate to fix the
implementation of code\_aster. However, this raises many questions. Are we
confident enough that this version really produces more accurate results? Should
results produced by ``older'' versions of code\_aster be considered invalid?

While 6~samples were enough for the purpose of debugging code\_aster, and in
light of what is at stake with a change in the actual implementation of
code\_aster, a higher degree of confidence should be required here.

Following the rules of Bernoulli experiments without failure and choosing
$p=1-\alpha=0.995$, the required number of samples is
$N=1058$. Table~\ref{cmpAsterEstimator} reports accuracy estimations of the
three versions of code\_aster for this new number of samples. Four estimators
are compared:
\begin{enumerate}
\item $\hat s_{\text{\sc b}}^{\hat\mu}$: estimator based on Bernoulli
  experiments without failure, with a reference result taken to be the average
  of all samples;
\item $\hat s_{\text{\sc b}}^{\text{\sc ieee}} $: estimator similar to the
  previous one, but taking the {\sc ieee} computation as reference;
\item $\hat s _{\textrm{\sc cnh}}$: estimator based on the Centered
  Normal Hypothesis.  Value $\hat \mu$ is taken as the reference to satisfy
  the centered hypothesis. This estimator should be used together with
  a normality test, such as Shapiro-Wilk, of which the $p$-value is displayed in parentheses.
\item $\hat s_{\textrm{\sc mca}}$: standard MCA estimator. 

\end{enumerate}

\medskip

\newcolumntype{P}[1]{>{\centering\arraybackslash}p{#1}}
\begin{table}
\centering\begin{subtable}{\textwidth}
    \centering
    \begin{tabular}{lccP{7em}c}
  \toprule
  Implementation & $\hat s_{\text{\sc b}}^{\hat\mu}$ & $\hat s_{\text{\sc b}}^{\text{\sc ieee}} $&  $\hat s_{\textrm{\sc cnh}}$ (normality test $p$-value) & $\hat s_{\textrm{\sc mca}}$ \\
 \midrule
 \texttt{version1}  &  28   & 28    & 29.01 (0.10) & 30.59\\
\texttt{version2}  &  29   & 29    & 29.55 (0.89) & 31.13\\
\texttt{version3}  &  30   & 31    & 31.22 (0.52) & 32.79\\
\bottomrule
\end{tabular}

    \caption{quantity $a$}
    \label{cmpAsterEstimator0}
    \end{subtable}
    \begin{subtable}{\textwidth}
    \centering
    \begin{tabular}{lccP{7em}c}
  \toprule
  Implementation & $\hat s_{\text{\sc b}}^{\hat\mu}$ & $\hat s_{\text{\sc b}}^{\text{\sc ieee}} $&  $\hat s_{\textrm{\sc cnh}}$ (normality test $p$-value) & $\hat s_{\textrm{\sc mca}}$ \\
 \midrule
 \texttt{version1}  &  17   & 17    & 17.85 (0.10) & 19.43\\
\texttt{version2}  &  18   & 18    & 18.39 (0.89) & 19.97\\
\texttt{version3}  &  19   & 19    & 20.05 (0.52) & 21.63\\
\bottomrule
\end{tabular}

    \caption{quantity $e$}
    \label{cmpAsterEstimator4}
    \end{subtable}

\caption{Comparison of stochastic estimators for 3 version of code\_aster, with 1058 samples.}
\label{cmpAsterEstimator}
\end{table} 
With the 4 estimators, \texttt{version3} is more accurate for both
variables. As the normality test sometimes fails, this version can be
selected only based on  $\hat s_{\text{\sc b}}^{\hat\mu}$ and $\hat
s_{\text{\sc b}}^{\text{\sc ieee}}$, which give similar results.

When results follow a Gaussian distribution\footnote{The normality of \texttt{version1} results might be subject to caution, but the normality of \texttt{version2}
and \texttt{version3} was not rejected with Shapiro-Wilk $p$-values significantly higher than~$0.1$.}, the estimate provided by $\hat
s_{\textrm{\sc cnh}}$ is slightly less conservative than~$\hat
s_{\text{b}}$ while remaining, by construction, sound and more
conservative than $\hat s_{\textrm{\sc mca}}$.

It is interesting to note here that the Bernoulli indicator
$\hat s_{\text{\sc b}}$ is by definition an integer, which can somewhat
be limiting. Looking for example at $\hat s_{\text{\sc b}}^{\hat\mu}$ for quantity
$e$, as computed by \texttt{version2} and \texttt{version3} and reported in
Table~\ref{cmpAsterEstimator4}, is the difference between 18 and 19~bits really
significant, or does it only come from definition~\eqref{eq:sb} restricting
$\hat s_{\text{\sc b}}$ to integer values? Using the centered normality
hypothesis estimator, the accuracy improvement between \texttt{version2} and
\texttt{version3} can be estimated to 1.66~bits. It is also possible to
generalize definition~\eqref{eq:sb} to fractional values of $k$. In this
example, this would yield estimated numbers of significant bits of 18.16 for
\texttt{version2} and 19.75 for \texttt{version3} respectively, again estimating
the gain in accuracy to approximately 1.6~bits.

\medskip

In any case, the results produced by \texttt{version3} are confidently estimated
to be computed with 20~reliable bits for quantity $e$ (\textit{i.e.} relative
error in the order of $2^{-20}\approx 9.5\times 10^{-7}$), which now satisfies
the user requirement of 6~decimal digits (although barely). Quantity $a$ is
estimated to be computed with at least 31~reliable bits, or
approximately~$4.7\times 10^{-10}$ relative error. And it is also safe
to conclude that this implementation is the most robust among all three tested
versions. \texttt{Version3} can thus be safely introduced in the code\_aster code
base.

One could also wonder, in retrospect, whether results produced by the IEEE
execution of the original version of code\_aster (\texttt{version0}) were
valid. The comparison of IEEE results produced by \texttt{version0} and
\texttt{version3} yields relative discrepancies of $4.29\times 10^{-10}$ and
$9.84\times 10^{-7}$ respectively for quantities $a$ and $e$. These
discrepancies have the same order of magnitude as the uncertainties estimated
above, an observation which does not invalidate the IEEE results produced by
\texttt{version0}.
 
\section{Conclusion and Future Works}
\label{sec:conclusion}
\label{sec:limitations}

Stochastic arithmetic methods like MCA or CESTAC suffer from some limitations. Some of these are addressed by the current work; some others are not and should therefore be addressed in future work. 

\medskip

First, stochastic arithmetic methods provide accuracy assessments which are valid only for the specific set of input data that were used at run-time. Similarly to other run-time based analysis, the robustness of the conclusions can only be obtained through the multiplication of test cases to maximize the coverage of the analysis. 

One way to overcome this limitation consists in using formal tools based on the static analysis of the source code. The aim of such tools is to prove that a given program follows its specifications for all possible input data in its admissible range. However, although much progress has been done in this area over the last few years~\cite{BCF13}, to the authors' knowledge, the use of such methods, especially in the field of scientific computing codes, is limited to small programs.

\medskip

Second, stochastic arithmetic methods do not provide a formal proof of correctness. Interval Arithmetic~\cite{rump2010verification} is an example of guaranteed run-time analysis method, which provides a sound interval for the result of a floating point computation. However, the use of such a method is not always tractable, especially for complex programs with data dependent control paths. In that case the solution often rapidly diverges to dramatic overestimation of the error and therefore does not bring useful information~\cite{martel2005}. Affine arithmetic can somewhat mitigate this effect and extend the applicability of guaranteed analyses to larger or more complex programs~\cite{ghorbal2009}.

On the other hand, stochastic arithmetic methods model round-off and cancellation errors with a Monte Carlo simulation. Stochastic methods scale well to programs with greater code size and complexity and do not suffer from the intractability problems of other. However, stochastic methods require multiple executions of the program, but the cost can be offset since they run in an embarrassingly parallel way. Their estimates are based on a limited number of random samples for each particular use-case input data. Therefore, as shown in this paper, stochastic methods do not provide formal guarantees but statistical confidence and probability.

\medskip

Finally, there remains an important limitation to stochastic arithmetic methods, even using our methodology: the practical usefulness of quality assessments provided by stochastic methods relies on the implied hypothesis that MCA or CESTAC correctly model the mechanisms causing accuracy losses in industrial calculations. 

If a stochastic arithmetic method fails to reproduce the effect of floating-point round-offs, threads scheduling changes or other important factors affecting the accuracy of the computed result, then the quality assessment will be affected not only by \textit{sampling} errors as seen above, but also by \textit{model} errors. CESTAC and MCA have been designed to model realistically the round-off and cancellation errors in IEEE-754 floating-point arithmetic. Practice shows they correctly simulate floating-points bugs in many cases. However, it is possible to find corner cases where CESTAC or MCA diverge from the IEEE-754 computation, producing results which can be more accurate~\cite{chesneaux1993algorithme} or less accurate~\cite{verificarlo}.

\medskip

Since the confidence intervals introduced in this paper tackle sampling errors, the next step consists in preventing and detecting model errors. This requires in-depth understanding of the limitations and hypotheses of each stochastic arithmetic model, documentation informing the practitioner of these hypotheses and, when possible, automated run-time checks raising warnings when these hypotheses are not met in a particular calculation. Some of the stochastic arithmetic software tools, such as CADNA, implement runtime checks that detect some cases violating the hypotheses of the model~\cite{JEZEQUEL2008933}.

In the meantime, model errors can be mitigated with simple sanity checks. For example, taking the IEEE-754 computation of the program as a reference value is a good sanity check. Indeed, if the stochastic model diverges from the IEEE-754 computation, the large error with respect to the reference value will raise a red flag.

The results of this paper are applicable to both centered normal distribution and general distributions. In the case of centered normal distributions, we show that to reach a given confidence level it is enough to apply a constant shift to the number of significant digits computed using the standard MCA estimator. In the case of general distribution, we provide a simple formula to decide how many samples are needed to reach a given confidence level. Appendix~\ref{sec:abaque} provides tables for the shifts and number of samples in the centered normal and general cases.

This paper also introduces a novel measure for the precision of a result, the number of contributing digits, after which digits are random noise. This metric can help choosing how many digits to keep when forwarding a result to another tool or storing a result during a checkpoint.

An accompanying Jupyter notebook~\footnote{\url{https://github.com/interflop/stochastic-confidence-intervals}} provides a reference Python implementation that can be used by the practitioner to compute the confidence intervals from this paper in the normal and Bernoulli cases.

\ifpreprint
\section*{Acknowledgements}
\else
\begin{acks}
\fi

The authors thank Yohan Chatelain for helpful reviews and feedbacks and the anonymous reviewers for their helpful comments.

\ifpreprint
\bibliographystyle{alpha}
\else
\end{acks}
\bibliographystyle{ACM-Reference-Format}
\fi
\bibliography{biblio}

\newcommand{\etalchar}[1]{$^{#1}$}
\begin{thebibliography}{CPdOC{\etalchar{+}}19}

\bibitem[Ado16]{sdnl112}
Andr{\'e} Adobes.
\newblock {Code\_Aster} : Sdnl112.
\newblock https://www.code-aster.org/V2/doc/v14/en/man\_v/v5/v5.02.112.pdf,
  2016.

\bibitem[BCD01]{brown2001interval}
Lawrence~D Brown, T~Tony Cai, and Anirban DasGupta.
\newblock Interval estimation for a binomial proportion.
\newblock {\em Statistical science}, pages 101--117, 2001.

\bibitem[BCF{\etalchar{+}}13]{BCF13}
Sylvie Boldo, Fran\c{c}ois Cl\'ement, Jean-Christophe Filli\^atre, Micaela
  Mayero, Guillaume Melquiond, and Pierre Weis.
\newblock {Wave Equation Numerical Resolution: a Comprehensive Mechanized Proof
  of a C Program}.
\newblock {\em Journal of Automated Reasoning}, 50(4):423--456, April 2013.

\bibitem[BFM09]{boldo2009combining}
Sylvie Boldo, Jean-Christophe Filli{\^a}tre, and Guillaume Melquiond.
\newblock Combining coq and gappa for certifying floating-point programs.
\newblock In {\em International Conference on Intelligent Computer
  Mathematics}, pages 59--74. Springer, 2009.

\bibitem[BM11]{boldo2011flocq}
Sylvie Boldo and Guillaume Melquiond.
\newblock Flocq: A unified library for proving floating-point algorithms in
  coq.
\newblock In {\em Computer Arithmetic (ARITH), 2011 20th IEEE Symposium on},
  pages 243--252. IEEE, 2011.

\bibitem[CdOCP{\etalchar{+}}18]{Chatelain2018veritracer}
Yohan Chatelain, Pablo de~Oliveira~Castro, Eric Petit, David Defour, Jordan
  Bieder, and Marc Torrent.
\newblock {VeriTracer: Context-enriched tracer for floating-point arithmetic
  analysis}.
\newblock In {\em 25th {IEEE} Symposium on Computer Arithmetic, {ARITH} 2018,
  Amherst, MA, USA. June 25th-27th, 2018}, page (to appear), 2018.

\bibitem[Cha88]{chatelin88}
Fran\c{c}oise Chatelin.
\newblock On the general reliability of the cestac method.
\newblock {\em C. R. Acad.Sci. Paris}, 1:851--854, 1988.

\bibitem[{Cod}18]{aster}
{Code\_Aster}.
\newblock Structures and thermomechanics analysis for studies and research.
\newblock http://www.code-aster.org/, 2018.

\bibitem[CPdOC{\etalchar{+}}19]{Chatelain2019automatic}
Yohan Chatelain, Eric Petit, Pablo de~Oliveira~Castro, Ghislain Lartigue, and
  David Defour.
\newblock Automatic exploration of reduced floating-point representations in
  iterative methods.
\newblock In {\em Euro-Par 2019 Parallel Processing - 25th International
  Conference}, Lecture Notes in Computer Science. Springer, 2019.

\bibitem[CV88]{chesneauxvignes}
Jean-Marie Chesneaux and Jean Vignes.
\newblock On the robustness of the cestac method.
\newblock {\em C. R. Acad.Sci. Paris}, 1:855--860, 1988.

\bibitem[CV93]{chesneaux1993algorithme}
Jean-Marie Chesneaux and Jean Vignes.
\newblock L'algorithme de gauss en arithm{\'e}tique stochastique.
\newblock {\em Comptes rendus de l'Acad{\'e}mie des sciences. S{\'e}rie 2,
  M{\'e}canique, Physique, Chimie, Sciences de l'univers, Sciences de la
  Terre}, 316(2):171--176, 1993.

\bibitem[DDLM06]{de2006assisted}
Florent De~Dinechin, Christoph~Quirin Lauter, and Guillaume Melquiond.
\newblock Assisted verification of elementary functions using gappa.
\newblock In {\em Proceedings of the 2006 ACM symposium on Applied computing},
  pages 1318--1322. ACM, 2006.

\bibitem[DdOCIP20]{Defour2020CustomPrecision}
David Defour, Pablo de~Oliveira~Castro, Matei Istoan, and Eric Petit.
\newblock Custom-precision mathematical library explorations for code profiling
  and optimization.
\newblock In {\em 27th {IEEE} Symposium on Computer Arithmetic, {ARITH} 2020},
  2020.

\bibitem[DdOCP16]{verificarlo}
Christophe Denis, Pablo de~Oliveira~Castro, and Eric Petit.
\newblock Verificarlo: Checking floating point accuracy through monte carlo
  arithmetic.
\newblock In {\em 23nd {IEEE} Symposium on Computer Arithmetic, {ARITH} 2016,
  Silicon Valley, CA, USA, July 10-13, 2016}, pages 55--62, 2016.

\bibitem[DTLJ01]{dessombz2001analysis}
Olivier Dessombz, Fabrice Thouverez, J-P La{\^\i}n{\'e}, and Louis
  J{\'e}z{\'e}quel.
\newblock Analysis of mechanical systems using interval computations applied to
  finite element methods.
\newblock {\em Journal of Sound and Vibration}, 239(5):949--968, 2001.

\bibitem[ELKT95]{Eypasch619}
Ernst Eypasch, Rolf Lefering, C~K Kum, and Hans Troidl.
\newblock Probability of adverse events that have not yet occurred: a
  statistical reminder.
\newblock {\em BMJ}, 311(7005):619--620, 1995.

\bibitem[{Eur}18]{EPXweb}
{Europlexus}.
\newblock Project web page, 2018.

\bibitem[FL15]{frechtling2014tool}
Michael Frechtling and Philip~H.W. Leong.
\newblock Mcalib: Measuring sensitivity to rounding error with monte carlo
  programming.
\newblock {\em ACM Transactions on Programming Languages and Systems}, 37(2):5,
  2015.

\bibitem[FL16]{fevotte2016}
Fran{\c c}ois F{\'e}votte and Bruno Lathuili{\`e}re.
\newblock {VERROU}: a {CESTAC} evaluation without recompilation.
\newblock In {\em International Symposium on Scientific Computing, Computer
  Arithmetics and Verified Numerics~(SCAN)}, Uppsala, Sweden, September 2016.

\bibitem[FL17]{fevotte2017studying}
Fran{\c{c}}ois F{\'e}votte and Bruno Lathuili{\`e}re.
\newblock Studying the numerical quality of an industrial computing code: A
  case study on code{\_}aster.
\newblock In {\em International Workshop on Numerical Software Verification},
  pages 61--80. Springer, 2017.

\bibitem[GGP09]{ghorbal2009}
Khalil Ghorbal, Eric Goubault, and Sylvie Putot.
\newblock The zonotope abstract domain taylor1+.
\newblock In {\em International Conference on Computer Aided Verification
  ({CAV})}, 2009.

\bibitem[GP06]{goubault2006static}
Eric Goubault and Sylvie Putot.
\newblock Static analysis of numerical algorithms.
\newblock In {\em International Static Analysis Symposium}, pages 18--34.
  Springer, 2006.

\bibitem[Han65]{hansen1965interval}
Eldon Hansen.
\newblock Interval arithmetic in matrix computations, part i.
\newblock {\em Journal of the Society for Industrial and Applied Mathematics,
  Series B: Numerical Analysis}, 2(2):308--320, 1965.

\bibitem[Hig02]{higham2002accuracy}
Nicholas~J. Higham.
\newblock {\em Accuracy and stability of numerical algorithms}.
\newblock Siam, 2002.

\bibitem[HJVE01]{hickey2001interval}
Timothy Hickey, Qun Ju, and Maarten~H Van~Emden.
\newblock Interval arithmetic: From principles to implementation.
\newblock {\em Journal of the ACM (JACM)}, 48(5):1038--1068, 2001.

\bibitem[JC08]{JEZEQUEL2008933}
Fabienne Jézéquel and Jean-Marie Chesneaux.
\newblock Cadna: a library for estimating round-off error propagation.
\newblock {\em Computer Physics Communications}, 178(12):933 -- 955, 2008.

\bibitem[Kah66]{KAHAN66}
William Kahan.
\newblock Numerical linear algebra.
\newblock In {\em Canadian Mathematical Bulletin}, pages 756--801, 1966.

\bibitem[Kah96]{kahan1996improbability}
William Kahan.
\newblock The improbability of probabilistic error analyses for numerical
  computations.
\newblock In {\em UCB Statistics Colloquium, Evans Hall edition}, page~20,
  1996.

\bibitem[LCJ10]{lamotte2010cadna_c}
Jean-Luc Lamotte, Jean-Marie Chesneaux, and Fabienne J{\'e}z{\'e}quel.
\newblock Cadna\_c: A version of cadna for use with c or c++ programs.
\newblock {\em Computer Physics Communications}, 181(11):1925--1926, 2010.

\bibitem[Li13]{Li2013Numerical}
Wenbin Li.
\newblock {\em Numerical accuracy analysis in simulations on hybrid
  high-performance computing systems}.
\newblock PhD thesis, University of Stuttgart, 2013.

\bibitem[Mar05]{martel2005}
Matthieu Martel.
\newblock An overview of semantics for the validation of numerical programs.
\newblock In {\em International Workshop on Verification, Model Checking, and
  Abstract Interpretation}, 2005.

\bibitem[MKC09]{moore2009introduction}
Ramon~E Moore, R~Baker Kearfott, and Michael~J Cloud.
\newblock {\em Introduction to interval analysis}, volume 110.
\newblock Siam, 2009.

\bibitem[Moo79]{moore1979methods}
Ramon~E Moore.
\newblock {\em Methods and applications of interval analysis}, volume~2.
\newblock Siam, 1979.

\bibitem[NS07]{nethercote2007}
Nicholas Nethercote and Julian Seward.
\newblock Valgrind: A framework for heavyweight dynamic binary instrumentation.
\newblock In {\em ACM SIGPLAN 2007 Conference on Programming Language Design
  and Implementation (PLDI)}, 2007.

\bibitem[ORO05]{ogita2005}
Takeshi Ogita, Siegfried~M. Rump, and Shin'ichi Oishi.
\newblock Accurate sum and dot product.
\newblock {\em SIAM J. Sci. Comput.}, 26:1955--1988, 2005.

\bibitem[RR05]{revol2005motivations}
Nathalie Revol and Fabrice Rouillier.
\newblock Motivations for an arbitrary precision interval arithmetic and the
  mpfi library.
\newblock {\em Reliable computing}, 11(4):275--290, 2005.

\bibitem[Rum99]{rump1999intlab}
Siegfried~M Rump.
\newblock Intlab—interval laboratory.
\newblock In {\em Developments in reliable computing}, pages 77--104. Springer,
  1999.

\bibitem[Rum10]{rump2010verification}
Siegfried~M Rump.
\newblock Verification methods: Rigorous results using floating-point
  arithmetic.
\newblock {\em Acta Numerica}, 19:287--449, 2010.

\bibitem[Sap11]{saporta2006probabilites}
Gilbert Saporta.
\newblock {\em Probabilit{\'e}s, analyse de donn{\'e}es et statistiques (3eme
  {\'e}dition)}.
\newblock Editions Technip, 2011.

\bibitem[SP97]{PARKER1997}
Douglas Stott~Parker.
\newblock Monte carlo arithmetic: exploiting randomness in floating-point
  arithmetic.
\newblock Technical Report CSD-970002, UCLA Computer Science Dept., 1997.

\bibitem[{Ver}18a]{verificarloproject}
{Verificarlo}.
\newblock Project repository, 2018.

\bibitem[{Ver}18b]{verrouproject}
{Verrou}.
\newblock Project repository, 2018.

\bibitem[Vig04]{VIGNES2004}
Jean Vignes.
\newblock Discrete stochastic arithmetic for validating results of numerical
  software.
\newblock {\em Numerical Algorithms}, 37(1-4):377--390, 2004.

\bibitem[VLP74]{VIGNES1974}
Jean Vignes and Michel La~Porte.
\newblock Error analysis in computing.
\newblock In {\em Proceedings of IFP 1974}, pages 610--614. IFP, 1974.

\bibitem[Zel09]{zeller2009}
Andreas Zeller.
\newblock {\em Why Programs Fail}.
\newblock Morgan Kaufmann, Boston, second edition, 2009.

\end{thebibliography}

\appendix\newpage
\section{Tabulated numbers of samples and variance shifts}
\label{sec:abaque}

This appendix presents tabulated values allowing the practitioner to easily
choose a desired number of samples (table~\ref{tab:nsamples}) in the case of a
general distribution of results (Bernoulli framework), or determine the shifting
levels to build sound estimators from the sampled variance in the centered
normality hypothesis (table~\ref{tab:shift}).

\begin{table}[bp]
\rowcolors{4}{black!10}{}
\begin{tabular}{cccccccccc}
  \toprule
  \multirow{2}{5em}{\centering Confidence level $1-\alpha$}
  & \multicolumn{9}{c}{Probability $p$} \\
  \cmidrule{2-10}
  & 0.66 & 0.75 & 0.8 & 0.85 & 0.9 & 0.95 & 0.99 & 0.995 & 0.999\\
  \midrule
  0.66&3 & 4 & 5 & 7 & 11 & 22 & 108 & 216 & 1079\\
0.75&4 & 5 & 7 & 9 & 14 & 28 & 138 & 277 & 1386\\
0.8&4 & 6 & 8 & 10 & 16 & 32 & 161 & 322 & 1609\\
0.85&5 & 7 & 9 & 12 & 19 & 37 & 189 & 379 & 1897\\
0.9&6 & 9 & 11 & 15 & 22 & 45 & 230 & 460 & 2302\\
0.95&8 & 11 & 14 & 19 & 29 & 59 & 299 & 598 & 2995\\
0.99&12 & 17 & 21 & 29 & 44 & 90 & 459 & 919 & 4603\\
0.995&13 & 19 & 24 & 33 & 51 & 104 & 528 & 1058 & 5296\\
0.999&17 & 25 & 31 & 43 & 66 & 135 & 688 & 1379 & 6905\\

  \bottomrule
\end{tabular}
\caption{Number of samples necessary to obtain a given confidence interval with
probability $p$, according to the Bernoulli estimator (\textit{i.e.} without
any assumption on the probability law).}
\label{tab:nsamples}
\end{table} 

\begin{landscape}
\begin{table}[h!]
\centering
\resizebox{57em}{!}{
\begin{tabular}{ccccccccccccccccccccc}
  \toprule
  & \multicolumn{1}{r}{$p$} & 0.66 & 0.66 & 0.75 & 0.75 & 0.75 & 0.9 & 0.9 & 0.9 & 0.95 & 0.95 & 0.95 & 0.99 & 0.99 & 0.99 & 0.995 & 0.995 & 0.995 & 0.999 & 0.999 \\
  & \multicolumn{1}{r}{$1-\alpha$} & 0.66 & 0.75 & 0.66 & 0.75 & 0.9 & 0.75 & 0.9 & 0.95 & 0.9 & 0.95 & 0.99 & 0.95 & 0.99 & 0.995 & 0.99 & 0.995 & 0.999 & 0.995 & 0.999 \\
  \midrule
  \multirow{36}{1em}{\rotatebox{90}{$N_{\text{samples}}$}}
  &3 & 1.145 & 1.385 & 1.415 & 1.655 & 2.345 & 2.171 & 2.861 & 3.370 & 3.114 & 3.623 & 4.791 & 4.017 & 5.186 & 5.687 & 5.310 & 5.811 & 6.972 & 6.040 & 7.202\\
&\cellcolor{black!10}4 & \cellcolor{black!10}0.817 & \cellcolor{black!10}0.990 & \cellcolor{black!10}1.086 & \cellcolor{black!10}1.260 & \cellcolor{black!10}1.749 & \cellcolor{black!10}1.776 & \cellcolor{black!10}2.264 & \cellcolor{black!10}2.617 & \cellcolor{black!10}2.517 & \cellcolor{black!10}2.870 & \cellcolor{black!10}3.665 & \cellcolor{black!10}3.264 & \cellcolor{black!10}4.059 & \cellcolor{black!10}4.396 & \cellcolor{black!10}4.183 & \cellcolor{black!10}4.520 & \cellcolor{black!10}5.298 & \cellcolor{black!10}4.749 & \cellcolor{black!10}5.527\\
&5 & 0.650 & 0.790 & 0.919 & 1.060 & 1.449 & 1.576 & 1.965 & 2.241 & 2.218 & 2.494 & 3.108 & 2.888 & 3.502 & 3.759 & 3.626 & 3.883 & 4.473 & 4.112 & 4.703\\
&\cellcolor{black!10}6 & \cellcolor{black!10}0.546 & \cellcolor{black!10}0.667 & \cellcolor{black!10}0.816 & \cellcolor{black!10}0.936 & \cellcolor{black!10}1.266 & \cellcolor{black!10}1.452 & \cellcolor{black!10}1.781 & \cellcolor{black!10}2.013 & \cellcolor{black!10}2.034 & \cellcolor{black!10}2.266 & \cellcolor{black!10}2.772 & \cellcolor{black!10}2.660 & \cellcolor{black!10}3.167 & \cellcolor{black!10}3.377 & \cellcolor{black!10}3.291 & \cellcolor{black!10}3.501 & \cellcolor{black!10}3.981 & \cellcolor{black!10}3.730 & \cellcolor{black!10}4.210\\
&8 & 0.423 & 0.519 & 0.693 & 0.789 & 1.048 & 1.305 & 1.564 & 1.744 & 1.817 & 1.997 & 2.383 & 2.391 & 2.777 & 2.935 & 2.901 & 3.059 & 3.415 & 3.288 & 3.645\\
&\cellcolor{black!10}9 & \cellcolor{black!10}0.382 & \cellcolor{black!10}0.470 & \cellcolor{black!10}0.652 & \cellcolor{black!10}0.740 & \cellcolor{black!10}0.977 & \cellcolor{black!10}1.256 & \cellcolor{black!10}1.493 & \cellcolor{black!10}1.656 & \cellcolor{black!10}1.746 & \cellcolor{black!10}1.909 & \cellcolor{black!10}2.258 & \cellcolor{black!10}2.303 & \cellcolor{black!10}2.652 & \cellcolor{black!10}2.794 & \cellcolor{black!10}2.776 & \cellcolor{black!10}2.918 & \cellcolor{black!10}3.236 & \cellcolor{black!10}3.147 & \cellcolor{black!10}3.465\\
&10 & 0.350 & 0.432 & 0.620 & 0.701 & 0.921 & 1.217 & 1.437 & 1.587 & 1.690 & 1.840 & 2.159 & 2.234 & 2.553 & 2.682 & 2.677 & 2.806 & 3.095 & 3.036 & 3.324\\
&\cellcolor{black!10}12 & \cellcolor{black!10}0.301 & \cellcolor{black!10}0.373 & \cellcolor{black!10}0.570 & \cellcolor{black!10}0.643 & \cellcolor{black!10}0.835 & \cellcolor{black!10}1.159 & \cellcolor{black!10}1.351 & \cellcolor{black!10}1.482 & \cellcolor{black!10}1.604 & \cellcolor{black!10}1.735 & \cellcolor{black!10}2.011 & \cellcolor{black!10}2.129 & \cellcolor{black!10}2.405 & \cellcolor{black!10}2.516 & \cellcolor{black!10}2.529 & \cellcolor{black!10}2.640 & \cellcolor{black!10}2.886 & \cellcolor{black!10}2.869 & \cellcolor{black!10}3.115\\
&14 & 0.265 & 0.330 & 0.535 & 0.600 & 0.773 & 1.116 & 1.289 & 1.406 & 1.542 & 1.659 & 1.905 & 2.054 & 2.299 & 2.397 & 2.423 & 2.521 & 2.737 & 2.750 & 2.967\\
&\cellcolor{black!10}15 & \cellcolor{black!10}0.250 & \cellcolor{black!10}0.313 & \cellcolor{black!10}0.520 & \cellcolor{black!10}0.583 & \cellcolor{black!10}0.748 & \cellcolor{black!10}1.099 & \cellcolor{black!10}1.264 & \cellcolor{black!10}1.376 & \cellcolor{black!10}1.517 & \cellcolor{black!10}1.629 & \cellcolor{black!10}1.862 & \cellcolor{black!10}2.023 & \cellcolor{black!10}2.256 & \cellcolor{black!10}2.349 & \cellcolor{black!10}2.380 & \cellcolor{black!10}2.473 & \cellcolor{black!10}2.678 & \cellcolor{black!10}2.702 & \cellcolor{black!10}2.907\\
&20 & 0.197 & 0.250 & 0.467 & 0.520 & 0.657 & 1.036 & 1.173 & 1.265 & 1.426 & 1.518 & 1.708 & 1.912 & 2.102 & 2.177 & 2.226 & 2.301 & 2.465 & 2.530 & 2.695\\
&\cellcolor{black!10}22 & \cellcolor{black!10}0.182 & \cellcolor{black!10}0.232 & \cellcolor{black!10}0.452 & \cellcolor{black!10}0.502 & \cellcolor{black!10}0.631 & \cellcolor{black!10}1.018 & \cellcolor{black!10}1.147 & \cellcolor{black!10}1.234 & \cellcolor{black!10}1.400 & \cellcolor{black!10}1.486 & \cellcolor{black!10}1.664 & \cellcolor{black!10}1.881 & \cellcolor{black!10}2.059 & \cellcolor{black!10}2.129 & \cellcolor{black!10}2.183 & \cellcolor{black!10}2.253 & \cellcolor{black!10}2.406 & \cellcolor{black!10}2.482 & \cellcolor{black!10}2.635\\
&25 & 0.164 & 0.210 & 0.434 & 0.479 & 0.599 & 0.995 & 1.115 & 1.195 & 1.368 & 1.448 & 1.611 & 1.842 & 2.005 & 2.070 & 2.129 & 2.194 & 2.333 & 2.423 & 2.562\\
&\cellcolor{black!10}29 & \cellcolor{black!10}0.144 & \cellcolor{black!10}0.186 & \cellcolor{black!10}0.414 & \cellcolor{black!10}0.456 & \cellcolor{black!10}0.566 & \cellcolor{black!10}0.972 & \cellcolor{black!10}1.081 & \cellcolor{black!10}1.154 & \cellcolor{black!10}1.334 & \cellcolor{black!10}1.407 & \cellcolor{black!10}1.555 & \cellcolor{black!10}1.801 & \cellcolor{black!10}1.950 & \cellcolor{black!10}2.008 & \cellcolor{black!10}2.074 & \cellcolor{black!10}2.132 & \cellcolor{black!10}2.257 & \cellcolor{black!10}2.361 & \cellcolor{black!10}2.487\\
&30 & 0.140 & 0.181 & 0.410 & 0.451 & 0.558 & 0.967 & 1.074 & 1.145 & 1.327 & 1.398 & 1.543 & 1.792 & 1.938 & 1.994 & 2.062 & 2.118 & 2.241 & 2.348 & 2.471\\
&\cellcolor{black!10}40 & \cellcolor{black!10}0.108 & \cellcolor{black!10}0.143 & \cellcolor{black!10}0.378 & \cellcolor{black!10}0.413 & \cellcolor{black!10}0.504 & \cellcolor{black!10}0.929 & \cellcolor{black!10}1.019 & \cellcolor{black!10}1.079 & \cellcolor{black!10}1.272 & \cellcolor{black!10}1.332 & \cellcolor{black!10}1.453 & \cellcolor{black!10}1.726 & \cellcolor{black!10}1.847 & \cellcolor{black!10}1.895 & \cellcolor{black!10}1.971 & \cellcolor{black!10}2.019 & \cellcolor{black!10}2.120 & \cellcolor{black!10}2.248 & \cellcolor{black!10}2.349\\
&45 & 0.097 & 0.129 & 0.366 & 0.399 & 0.484 & 0.915 & 1.000 & 1.056 & 1.253 & 1.308 & 1.421 & 1.703 & 1.815 & 1.859 & 1.939 & 1.983 & 2.077 & 2.212 & 2.306\\
&\cellcolor{black!10}50 & \cellcolor{black!10}0.087 & \cellcolor{black!10}0.118 & \cellcolor{black!10}0.357 & \cellcolor{black!10}0.388 & \cellcolor{black!10}0.468 & \cellcolor{black!10}0.904 & \cellcolor{black!10}0.984 & \cellcolor{black!10}1.036 & \cellcolor{black!10}1.236 & \cellcolor{black!10}1.289 & \cellcolor{black!10}1.395 & \cellcolor{black!10}1.683 & \cellcolor{black!10}1.789 & \cellcolor{black!10}1.830 & \cellcolor{black!10}1.913 & \cellcolor{black!10}1.954 & \cellcolor{black!10}2.042 & \cellcolor{black!10}2.183 & \cellcolor{black!10}2.271\\
&59 & 0.073 & 0.102 & 0.343 & 0.371 & 0.444 & 0.887 & 0.960 & 1.008 & 1.213 & 1.261 & 1.356 & 1.655 & 1.751 & 1.788 & 1.875 & 1.912 & 1.991 & 2.141 & 2.220\\
&\cellcolor{black!10}75 & \cellcolor{black!10}0.056 & \cellcolor{black!10}0.081 & \cellcolor{black!10}0.326 & \cellcolor{black!10}0.350 & \cellcolor{black!10}0.414 & \cellcolor{black!10}0.866 & \cellcolor{black!10}0.930 & \cellcolor{black!10}0.972 & \cellcolor{black!10}1.183 & \cellcolor{black!10}1.224 & \cellcolor{black!10}1.308 & \cellcolor{black!10}1.619 & \cellcolor{black!10}1.702 & \cellcolor{black!10}1.734 & \cellcolor{black!10}1.826 & \cellcolor{black!10}1.858 & \cellcolor{black!10}1.927 & \cellcolor{black!10}2.087 & \cellcolor{black!10}2.156\\
&90 & 0.044 & 0.067 & 0.314 & 0.336 & 0.394 & 0.852 & 0.910 & 0.947 & 1.163 & 1.200 & 1.275 & 1.595 & 1.670 & 1.698 & 1.794 & 1.822 & 1.884 & 2.052 & 2.113\\
&\cellcolor{black!10}100 & \cellcolor{black!10}0.038 & \cellcolor{black!10}0.059 & \cellcolor{black!10}0.308 & \cellcolor{black!10}0.329 & \cellcolor{black!10}0.383 & \cellcolor{black!10}0.845 & \cellcolor{black!10}0.899 & \cellcolor{black!10}0.935 & \cellcolor{black!10}1.152 & \cellcolor{black!10}1.188 & \cellcolor{black!10}1.258 & \cellcolor{black!10}1.582 & \cellcolor{black!10}1.652 & \cellcolor{black!10}1.680 & \cellcolor{black!10}1.776 & \cellcolor{black!10}1.804 & \cellcolor{black!10}1.861 & \cellcolor{black!10}2.033 & \cellcolor{black!10}2.091\\
&200 & 0.005 & 0.020 & 0.275 & 0.290 & 0.327 & 0.806 & 0.843 & 0.868 & 1.096 & 1.120 & 1.169 & 1.515 & 1.563 & 1.581 & 1.687 & 1.705 & 1.744 & 1.935 & 1.973\\
&\cellcolor{black!10}299 & \cellcolor{black!10}-0.008 & \cellcolor{black!10}0.004 & \cellcolor{black!10}0.261 & \cellcolor{black!10}0.273 & \cellcolor{black!10}0.304 & \cellcolor{black!10}0.789 & \cellcolor{black!10}0.820 & \cellcolor{black!10}0.839 & \cellcolor{black!10}1.072 & \cellcolor{black!10}1.092 & \cellcolor{black!10}1.131 & \cellcolor{black!10}1.486 & \cellcolor{black!10}1.525 & \cellcolor{black!10}1.540 & \cellcolor{black!10}1.649 & \cellcolor{black!10}1.664 & \cellcolor{black!10}1.695 & \cellcolor{black!10}1.893 & \cellcolor{black!10}1.924\\
&300 & -0.008 & 0.003 & 0.261 & 0.273 & 0.304 & 0.789 & 0.819 & 0.839 & 1.072 & 1.092 & 1.131 & 1.486 & 1.525 & 1.540 & 1.649 & 1.664 & 1.695 & 1.893 & 1.924\\
&\cellcolor{black!10}459 & \cellcolor{black!10}-0.020 & \cellcolor{black!10}-0.011 & \cellcolor{black!10}0.250 & \cellcolor{black!10}0.259 & \cellcolor{black!10}0.284 & \cellcolor{black!10}0.775 & \cellcolor{black!10}0.799 & \cellcolor{black!10}0.815 & \cellcolor{black!10}1.052 & \cellcolor{black!10}1.068 & \cellcolor{black!10}1.099 & \cellcolor{black!10}1.462 & \cellcolor{black!10}1.493 & \cellcolor{black!10}1.505 & \cellcolor{black!10}1.617 & \cellcolor{black!10}1.629 & \cellcolor{black!10}1.653 & \cellcolor{black!10}1.858 & \cellcolor{black!10}1.883\\
&500 & -0.022 & -0.013 & 0.248 & 0.257 & 0.280 & 0.773 & 0.796 & 0.811 & 1.049 & 1.064 & 1.093 & 1.458 & 1.487 & 1.499 & 1.611 & 1.623 & 1.646 & 1.852 & 1.876\\
&\cellcolor{black!10}528 & \cellcolor{black!10}-0.023 & \cellcolor{black!10}-0.015 & \cellcolor{black!10}0.246 & \cellcolor{black!10}0.255 & \cellcolor{black!10}0.278 & \cellcolor{black!10}0.771 & \cellcolor{black!10}0.794 & \cellcolor{black!10}0.808 & \cellcolor{black!10}1.047 & \cellcolor{black!10}1.061 & \cellcolor{black!10}1.090 & \cellcolor{black!10}1.455 & \cellcolor{black!10}1.484 & \cellcolor{black!10}1.495 & \cellcolor{black!10}1.608 & \cellcolor{black!10}1.619 & \cellcolor{black!10}1.642 & \cellcolor{black!10}1.848 & \cellcolor{black!10}1.871\\
&750 & -0.031 & -0.023 & 0.239 & 0.247 & 0.265 & 0.762 & 0.781 & 0.793 & 1.034 & 1.046 & 1.070 & 1.441 & 1.464 & 1.473 & 1.588 & 1.597 & 1.616 & 1.827 & 1.846\\
&\cellcolor{black!10}919 & \cellcolor{black!10}-0.034 & \cellcolor{black!10}-0.028 & \cellcolor{black!10}0.235 & \cellcolor{black!10}0.242 & \cellcolor{black!10}0.259 & \cellcolor{black!10}0.758 & \cellcolor{black!10}0.775 & \cellcolor{black!10}0.786 & \cellcolor{black!10}1.028 & \cellcolor{black!10}1.039 & \cellcolor{black!10}1.060 & \cellcolor{black!10}1.433 & \cellcolor{black!10}1.455 & \cellcolor{black!10}1.463 & \cellcolor{black!10}1.579 & \cellcolor{black!10}1.587 & \cellcolor{black!10}1.604 & \cellcolor{black!10}1.816 & \cellcolor{black!10}1.833\\
&1000 & -0.036 & -0.029 & 0.234 & 0.241 & 0.257 & 0.756 & 0.773 & 0.783 & 1.026 & 1.036 & 1.057 & 1.430 & 1.451 & 1.459 & 1.575 & 1.583 & 1.599 & 1.812 & 1.828\\
&\cellcolor{black!10}1058 & \cellcolor{black!10}-0.037 & \cellcolor{black!10}-0.030 & \cellcolor{black!10}0.233 & \cellcolor{black!10}0.239 & \cellcolor{black!10}0.255 & \cellcolor{black!10}0.755 & \cellcolor{black!10}0.771 & \cellcolor{black!10}0.781 & \cellcolor{black!10}1.024 & \cellcolor{black!10}1.034 & \cellcolor{black!10}1.054 & \cellcolor{black!10}1.428 & \cellcolor{black!10}1.448 & \cellcolor{black!10}1.456 & \cellcolor{black!10}1.572 & \cellcolor{black!10}1.580 & \cellcolor{black!10}1.596 & \cellcolor{black!10}1.809 & \cellcolor{black!10}1.825\\
&1379 & -0.040 & -0.035 & 0.229 & 0.235 & 0.249 & 0.751 & 0.764 & 0.773 & 1.017 & 1.026 & 1.044 & 1.420 & 1.438 & 1.444 & 1.562 & 1.568 & 1.582 & 1.798 & 1.812\\
&\cellcolor{black!10}5296 & \cellcolor{black!10}-0.054 & \cellcolor{black!10}-0.051 & \cellcolor{black!10}0.216 & \cellcolor{black!10}0.219 & \cellcolor{black!10}0.226 & \cellcolor{black!10}0.735 & \cellcolor{black!10}0.742 & \cellcolor{black!10}0.746 & \cellcolor{black!10}0.995 & \cellcolor{black!10}0.999 & \cellcolor{black!10}1.008 & \cellcolor{black!10}1.393 & \cellcolor{black!10}1.402 & \cellcolor{black!10}1.405 & \cellcolor{black!10}1.526 & \cellcolor{black!10}1.529 & \cellcolor{black!10}1.536 & \cellcolor{black!10}1.759 & \cellcolor{black!10}1.766\\
&6905 & -0.055 & -0.053 & 0.214 & 0.217 & 0.223 & 0.733 & 0.739 & 0.743 & 0.992 & 0.996 & 1.003 & 1.390 & 1.397 & 1.400 & 1.521 & 1.524 & 1.530 & 1.754 & 1.760\\
&\cellcolor{black!10}10000 & \cellcolor{black!10}-0.057 & \cellcolor{black!10}-0.055 & \cellcolor{black!10}0.212 & \cellcolor{black!10}0.214 & \cellcolor{black!10}0.219 & \cellcolor{black!10}0.730 & \cellcolor{black!10}0.735 & \cellcolor{black!10}0.739 & \cellcolor{black!10}0.988 & \cellcolor{black!10}0.991 & \cellcolor{black!10}0.998 & \cellcolor{black!10}1.386 & \cellcolor{black!10}1.392 & \cellcolor{black!10}1.394 & \cellcolor{black!10}1.516 & \cellcolor{black!10}1.518 & \cellcolor{black!10}1.523 & \cellcolor{black!10}1.748 & \cellcolor{black!10}1.753\\

  \bottomrule
\end{tabular}}
\medskip

\caption{Tabulated values of the shift~$\delta_{\text{\sc cnh}}$ used to
compute sound lower bounds in the Centered Normality Hypothesis, for various
sets of parameters.}

\label{tab:shift}
\end{table}
\end{landscape}

\end{document}